\documentclass[a4paper, 10pt]{article}
\usepackage{latexsym}
\usepackage{amsmath}
\usepackage{amsthm}
\usepackage{amssymb}
\usepackage[all]{xy}

\numberwithin{equation}{section}

\newtheorem{thm}[equation]{Theorem}

\newtheorem{lem}[equation]{Lemma}
\newtheorem{cor}[equation]{Corollary}

\theoremstyle{definition}

\theoremstyle{remark}
\newtheorem{rem}[equation]{Remark}
\newtheorem{eg}[equation]{Example}
\newtheorem{alg}[equation]{Algorithm}

\newtheorem{claim}{Claim}

\newcommand{\sk}{\vspace{\baselineskip}}

\title{Homogeneous coloured multipartite graphs}
\author{Deborah C Lockett \\
Department of Pure Mathematics, University of Leeds, \\
Leeds LS2 9JT, UK, d.lockett@outlook.com \\
and}   
\date{John K Truss \\
Department of Pure Mathematics, University of Leeds, \\
Leeds LS2 9JT, UK, pmtjkt@leeds.ac.uk, 
corresponding author, (tel. 01133435128)$^1$}{
\begin{document}
\maketitle 

\setcounter{footnote}{1}\footnotetext{Supported by EPSRC grant EP/H00677X/1.}
\newcounter{number}

\section*{Abstract}

We classify the countable homogeneous coloured multipartite graphs with any finite number of parts. By Fra\"iss\'e's Theorem this amounts to 
classifying the families ${\cal F}$ of pairwise non-embeddable finite coloured multipartite graphs for which the class $Forb({\cal F})$ of 
multipartite graphs which forbid these is an amalgamation class. We show that once we understand such families ${\cal F}$ in the quadripartite case, 
things do not become any more complicated for larger numbers of parts.

\section{Introduction}

The paper makes a contribution towards the classification of certain types of countable homogeneous structures, addressing a special case of a problem 
mentioned in \cite{cherlin}. There is now a large body of results giving classifications of countable homogeneous relational structures. For instance, 
the classes of countable homogeneous (undirected) graphs, tournaments, directed graphs, partial orders, and coloured partial orders were classified in 
\cite{lachlan1}, \cite{lachlan2}, \cite{cherlin}, \cite{schmerl}, \cite{torrezao} respectively. Cherlin suggested that a natural next step in the 
spirit of these earlier classifications would be to look at the countable homogeneous $n$-graphs for positive integers $n$, relational structures 
whose domain is the disjoint union of $n$ parts, on each of which there is an (ordinary) undirected graph, and between any two of which the edges are 
coloured by a (fixed) finite set of colours. 

Special cases of the question were addressed in \cite{jenkinson2} and \cite{rose}. The former case treated just `multipartite' graphs, being those in
which there are no edges within the parts, and only edges or non-edges in between, generalizing the usual notion of `bipartite'. In the latter, $n$ 
was taken to be 2, but arbitrary graphs on the two parts were allowed. In this case it is easy to see that if the whole structure is homogeneous, then 
so are the graphs on the two parts, which must therefore lie in Lachlan and Woodrow's list \cite{lachlan1}. Even here, the work is incomplete, and 
only certain possibilities for these two parts have so far been covered.  

More formally, an \emph{$n$-graph} is a graph on $n$ pairwise disjoint sets of vertices called \emph{parts} each of which is an ordinary graph, with 
finitely many possible edge-types between pairs of parts, which we think of as colours for these edges. Here we shall only consider the 
\emph{multipartite} case, where the parts are all null graphs. We label the parts of an \emph{$m$-partite graph} $G$ as $V_0, V_1, \ldots, V_{m-1}$. 
By a \emph{restriction} of a multipartite graph $G$ we mean an induced subgraph of $G$ whose vertex set is a union of some set of parts of $G$. For distinct $i, j$, we refer to the bipartite restriction of $G$ to two parts $V_i$ and $V_j$ as $V_iV_j$, and $E_{ij} := V_i \times V_j$ is the set of edges 
in $V_iV_j$. The case of `ordinary' bipartite graphs is the special case where there are just two colours, `joined' and `not joined'. We often use the 
shorthand notation $xy$ for an edge $(x,y)$. The \emph{edge-types} of $E_{ij}$ are given by the sets of colours $C_{ij}$, where $|C_{ij}| < \aleph_0$ 
for each distinct $i,j$. Colours are assigned to the edges by the functions $F_{ij}: E_{ij} \to C_{ij}$, which we may assume to be surjective. Let $F$ 
be the union of all of these colouring functions for each of the bipartite restrictions; so if $x \in V_i, ~y \in V_j$, then $F(x,y):= F_{ij}((x,y))$.

For $A$ a subgraph of $G$, let $V_i^A$ stand for the set of vertices of $A$ in part $V_i$, and so on. We use the convention that vertices will be labelled by a subscript to indicate which part they are in, for example $v_0 \in V_0$, $x_i \in V_i$. For a subset $A$ of the vertices of $G$, we use $\langle A \rangle$ to denote the induced subgraph on $A$, however if the meaning is clear we often just refer to $A$ itself as the subgraph. 

In model-theoretic terms, we are working in a finite relational language $L$, where we have finitely many unary relations which are required to form a partition of $V$, and other binary relations corresponding to each of the possible colours for edges between vertices in different parts (there are a finite number of these, because we only consider graphs with finitely many parts, and finitely many edge-types/colours between each pair of parts). We refer to coloured multipartite graphs in the language $L$ as \emph{$L$-graphs}.

A relational structure $M$ is said to be \emph{homogeneous} if every isomorphism between finite substructures of $M$ extends to an automorphism of 
$M$. Our ultimate aim is a classification of the countable homogeneous coloured multipartite graphs. That is, we aim to classify the countable homogeneous models of each finite language for multipartite graphs (where the language determines the number of parts, and the number of edge-types/colours for each pair of parts, as described above). 

We use the notion of \emph{substructure} as in the model-theoretic sense, that is, a substructure of $M$ is a subset of the domain of $M$ with all relations inherited from $M$ (like `induced subgraph' for graphs). An \emph{isomorphism} between $L$-structures $M_1,M_2$ is a bijection $\phi$ such that for each $k$-ary relation $R$ in $L$, and $x_1, \ldots, x_k \in M_1$, $R(x_1, \ldots, x_k)$ holds in $M_1$ if and only if $R(\phi(x_1), \ldots, \phi(x_k))$ holds in $M_2$. An \emph{embedding} of $A$ in $M$ is an isomorphism between $A$ and a substructure $A'$ of $M$. 
We refer to an isomorphism between finite substructures of $M$ as a \emph{finite partial automorphism} of $M$. 

In particular, note that if $\phi$ is a partial automorphism of the multipartite graph $G$, then the parts are fixed by $\phi$ (that is, $x \in {\rm dom} \phi \cap V_i$ if and only if $\phi(x) \in V_i$), since there is a relation in the language which determines which part each vertex is in, and this relation must be preserved by any isomorphism.

\subsection{Preliminary results} \label{prelim}

This project follows the work of Jenkinson, Seidel and Truss (see \cite{jenkinson2} which combines and extends results from \cite{jenkinson1} and 
\cite{seidel}). They mainly concentrated on the non-coloured case, that is, where each bipartite restriction is just an ordinary non-coloured graph. 
With our set-up, this actually corresponds to the case where we have (at most) two colours between each pair of parts, the `colours' being `joined' (edges) and `not joined' (non-edges). In \cite{jenkinson2}, a complete characterization of the ordinary non-coloured countable homogeneous multipartite graphs was obtained, and we aim to build on these results looking at the more general coloured case. 

In fact, coloured multipartite graphs were considered in \cite{jenkinson2}, but only for graphs with just two parts. It was found that there are only a few simple possibilities in this case, and we begin by stating this classification result for the homogeneous coloured bipartite graphs. 

\begin{thm}[\cite{jenkinson2}] \label{bipartite}
If $G$ is a countable homogeneous $C$-coloured bipartite graph where $1 \le |C| < \aleph_0$, then one of the following holds:
\begin{enumerate}
\item[(i)] $|C|=1$ and all the edges have the same colour;
\item[(ii)] $|C|=2$ and the edges of one colour are a perfect matching, and those of the other colour are its complement;
\item[(iii)] $|C| \ge 2$ and $G$ is $C$-generic.
\end{enumerate}
\end{thm}

Here, we say that the $C$-coloured bipartite graph $G$ is \emph{generic} (or \emph{$C$-generic}) if both parts are countably infinite, and for any 
finite subset $U \in V_i$ ($i = 0$ or $1$) and map $f: U \to C$, there is $x \in V_{1-i}$ such that $F(x,u) = f(u)$ for each $u \in U$. 

Note that in case (i) the parts can each be any (finite or) countable size, in case (ii) the parts must have the same countable size, and in case (iii) the parts 
are both countably infinite.

As well as being the natural starting point for the classification of general homogeneous multipartite graphs, this result is in fact also the basis 
for the classification of graphs with more parts, in view of the following result. 

\begin{lem}[\cite{jenkinson2}] \label{restriction}
Any restriction of a homogeneous multipartite graph to a subset of its set of parts is also homogeneous.
\end{lem}

\begin{proof} 
Let $G$ be a homogeneous multipartite graph, $G'$ a restriction of $G$, and $\phi$ a finite partial automorphism of $G'$. Then $\phi$ is also a finite 
partial automorphism of $G$, and since $G$ is homogeneous this extends to an automorphism $\psi$ of $G$. Since any automorphism preserves the parts, 
the restriction of $\psi$ to $G'$ is an automorphism of $G'$, which extends $\phi$. Hence $G'$ is also homogeneous. 
\end{proof}

In classifying ordinary graphs, we only work up to isomorphism. To simplify further, for homogeneous graphs we may choose to work up to anti-isomorphism (that is, taking the complement of a graph, interchanging edges and non-edges), since a graph is homogeneous if and only if its complement is homogeneous. Correspondingly, for coloured multipartite graphs we only want to work up to \emph{colour-isomorphism}.

For $k = 0,1$, let $G^k$ be a coloured $m$-partite graph on parts $V_0^k, \ldots, V_{m-1}^k$, where $F_{ij}^k: E_{ij}^k \to C_{ij}^k$ is the colouring function for the bipartite restriction $V_i^kV_j^k$ (for each distinct $i,j$). 
We say that $G^0, G^1$ are \emph{colour-isomorphic} if there is a bijection $\theta: \{0, \ldots, m-1\} \to \{0, \ldots, m-1\}$, such that there are bijections $\sigma_{ij} : C_{ij}^0 \to C_{\theta(i)\theta(j)}^1$ for each distinct $i,j$,
and bijections $\tau_i: V_i^0 \to V_{\theta(i)}^1$ for each $i \in \{0, \ldots, m-1\}$, 
such that $\sigma_{ij}(F_{ij}^0(x,y)) = F_{\theta(i)\theta(j)}^1(\tau_i(x), \tau_j(y))$ for each $x \in V_i^0, ~y \in V_j^0$. 
Informally, this says that there is an isomorphism between $G^0, G^1$ up to some relabelling of the colours.

Observe that to classify homogeneous coloured multipartite graphs, it is sufficient to work up to colour-isomorphism. Basically this says that relabelling the colours of a homogeneous multipartite graph does not affect it. 

\begin{lem} 
Let $G^0, G^1$ be colour-isomorphic multipartite graphs. Then $G^0$ is homogeneous if and only if $G^1$ is homogeneous.
\end{lem}

\begin{proof}
Similar to Lemma 1.2 from \cite{jenkinson2}, this is straightforward from the definition of colour-isomorphic.
\end{proof}

Next we see how (as in \cite{jenkinson2}) the problem of classifying homogeneous multipartite graphs may be reduced to classifying those for which all bipartite restrictions are generic. We write $G - V_i$ (for example) to denote the set-theoretic difference.
 
\begin{lem}  \label{perfectm}
Suppose that $G$ is a multipartite graph and that parts $V_i$ and $V_j$ are related by a perfect matching. Then $G$ is homogeneous if and only if $G - V_j$ is homogeneous and the map from $G - V_j$ to $G - V_i$ induced by the perfect matching is a colour-isomorphism.
\end{lem}

\begin{proof}
This is a straightforward generalization of Lemma 1.3 from \cite{jenkinson2} to the coloured case.
\end{proof}

Thus if $G$ is a homogeneous multipartite graph with a perfect matching on $V_iV_j$, then characterizing $G$ reduces to characterizing $G - V_i$. Given this strong condition for the occurrence of perfect matchings between parts in homogeneous multipartite graphs, it suffices to concentrate on graphs without perfect matchings on any bipartite restriction.

Now we may observe that if some part $V_i$ of $G$ was finite, then each bipartite restriction of which it was a part would have to be complete in one colour. As before, characterizing $G$ reduces to characterizing $G - V_i$, so we may reduce the problem to only considering graphs with no finite parts. Also note that we may view a bipartite restriction on two infinite parts with all edges the same colour as trivially $C$-generic where $|C|=1$. So from now on we may assume that all parts are countably infinite and all bipartite restrictions are generic. We call such a countable $m$-partite graph \emph{$m$-generic}. 

Furthermore, we may assume that our homogeneous multipartite graphs are not already covered by the classification in \cite{jenkinson2}, that is, some bipartite restriction has at least three edge-types/colours.

\subsection{Fra\" iss\' e's Theorem and amalgamation classes}

Let us now give further details of some relevant model-theoretic notions. 
For this we work in a general setting, in a fixed countable relational language $L$ (and note that in the case of coloured multipartite graphs, $L$ is in fact finite). 

An \emph{amalgamation class} is a family ${\cal C}$ of finite $L$-structures which is closed under isomorphism and taking substructures, and which has 
the \emph{amalgamation property}: if $A, B_1, B_2 \in {\cal C}$ and $f_i : A \to B_i ~(i = 1,2)$ are embeddings, then $B_1, B_2$ can be 
\emph{amalgamated over $A$}, that is, there are $C \in {\cal C}$ and embeddings $g_i : B_i \to C ~(i = 1,2)$ such that $g_1 \circ f_1 = g_2 \circ f_2$. 
We say that $C$ is an \emph{amalgam} of $B_1, B_2$ over $A$. 

For a countable $L$-structure $M$, the \emph{age} of $M$, $Age(M)$, is defined to be the family of finite $L$-structures which can be embedded in $M$. 

\begin{thm}[Fra\" iss\' e's Theorem] \label{fraissethm}
If $M$ is a countable homogeneous structure in a countable language, then $Age(M)$ is an amalgamation class. Conversely, if ${\cal C}$ is an amalgamation class of finite $L$-structures, then there is a countable homogeneous $L$-structure $M$, unique up to isomorphism, with $Age(M) = {\cal C}$.
\end{thm}

Now let us return to the specific case where $L$ is a language for coloured multipartite graphs, and let $G, A$ be $L$-graphs. If a finite graph $A$ can be embedded in $G$ then we say it is \emph{realized} in $G$, otherwise it is \emph{omitted} (or \emph{forbidden}). 
Futhermore, $A$ is \emph{minimally omitted} if it is omitted and every proper induced subgraph of $A$ is realized in $G$. 
For $G$ a homogeneous $m$-generic graph, $O(G)$ is defined to be the class of finite graphs minimally omitted from $G$. 

\begin{lem} \label{forbfam}
Suppose $G_0, G_1$ are countable homogeneous multipartite graphs in the same language $L$. Then $G_0, G_1$ are isomorphic if and only if they minimally omit the same class of finite multipartite $L$-graphs. 
\end{lem}

\begin{proof}
This is a straightforward generalization of Lemma 1.4 from \cite{jenkinson2}.
\end{proof}

We use $Forb({\cal F})$ to denote the class of all finite $L$-graphs which omit all members of a family ${\cal F}$ of finite $L$-graphs. Classifying 
the countable homogeneous coloured multipartite graphs then amounts to describing all families of pairwise non-embeddable $L$-graphs ${\cal F}$ for 
which $Forb({\cal F})$ is an amalgamation class. For if we note that for any $L$-graph $G$, $Forb(O(G)) = Age(G)$ (as one sees by considering for 
finite $A \not \in Age(G)$, a minimally omitted subgraph of $A$), Fra\" iss\'e's Theorem says: if $G$ is a countable homogeneous $L$-graph, then 
$Forb(O(G))$ is an amalgamation class; and conversely, if ${\cal F}$ is a family of finite $L$-graphs such that $Forb({\cal F})$ is an amalgamation 
class, then there is a countable homogeneous $m$-generic $L$-graph $G$ (unique up to isomorphism) with $Age(G) = Forb({\cal F})$ (and if the members 
of ${\cal F}$ are pairwise non-embeddable, then $O(G) = {\cal F}$).

To verify that a class of $L$-graphs is an amalgamation class, it will suffice to show that we can always perform `two-point amalgamations' (since we 
can then simply repeat a finite number of times to amalgamate more points). That is, we can always amalgamate $B_1, B_2$ over $A$ for $B_1, B_2$ which 
each have only one more point than $A$. In practice we also usually assume that $A$ is a substructure of $B_1$ and $B_2$, and that $A = B_1 \cap B_2$,
which can always be achieved by taking isomorphic copies.

\section{Non-monic realization}

From now on, we work in a fixed finite language $L$ for coloured $m$-partite graphs, and let $G$ be a countable homogeneous 
$m$-generic $L$-graph. In this section we show that, as in the ordinary non-coloured case in \cite{jenkinson2}, we are able to 
deduce that all members of $O(G)$ are monic, where a multipartite graph is said to be \emph{monic} if it has at most one 
vertex in each part. Correspondingly, a multipartite graph is \emph{non-monic} if there is some part in which it has more than 
one vertex. Note that a monic graph only has one edge in each bipartite restriction on which it is defined. To describe the 
colours of the edges in a monic graph $A$, we extend the notation for the colouring function. So if 
$v_i \in V^A_i, ~v_j \in V^A_j$, then we define $F_A(E_{ij}^A):= F(v_iv_j)$, which is often just written $F(E_{ij})$ provided 
$A$ is clear. Thus we suppress mention of $A$ when no confusion arises, both as a subscript on the map $F$, and as a
superscript on the edge set $E_{ij}$ (which in the restriction only has one member).

\begin{rem} \label{A>2}
We may assume that each minimally omitted graph is defined on at least three parts. For if there was some $A \in O(G)$ with $|A|=2$, then this corresponds to having one fewer colour in the language $L$ on this bipartite restriction of $G$. So we may assume that there is no such minimally omitted graph of size two, and all colours appear. By Lemma~\ref{restriction}, each bipartite restriction ${V_iV_j}$ is homogeneous, and so by   Theorem~\ref{bipartite} it is generic. Since the age of a countable generic graph is the family of all finite graphs in the corresponding language, there are no omitted graphs defined on only two parts. 
\end{rem}

The `non-monic realization theorem' is as follows:

\begin{thm} \label{nonmonic}
If $G$ is a homogeneous $m$-generic $L$-graph, then each $A \in O(G)$ is monic.
\end{thm}

The theorem is proved in the next subsection. Following this is an explanation of how the theorem is sufficient to give an effective (if not explicit) solution to the problem of determining all the countable homogeneous $m$-partite $L$-graphs in some specified language $L$.

\subsection{Proving the non-monic realization theorem}

We move towards proving Theorem~\ref{nonmonic}. 
First, we may easily show that for each member of $O(G)$, every bipartite restriction has edges of only one colour.

\begin{lem} \label{monobip}
If $G$ is a homogeneous $m$-generic $L$-graph and $A \in O(G)$, then each bipartite restriction of $A$ is monochromatic.
\end{lem}

\begin{proof}
Suppose $V_i^A V_j^A$ is not monochromatic. Then there are $x_i, y_i \in V_i^A$, $x_j, y_j \in V_j^A$ where $F(x_i, x_j) = \alpha$, 
$F(y_i, y_j) = \beta$ for distinct $\alpha, \beta \in C_{ij}$. Then we can find a vertex in one of the parts incident with edges of distinct colours. 
If this condition does not hold for $x_i$, then $F(x_i, z_j) = \alpha$ for each $z_j \in V_j^A$, and $y_j \in V_j^A$ is joined to $x_i$ by an 
$\alpha$-coloured edge, and to $y_i$ by a $\beta$-coloured edge. 

Without loss of generality, suppose there are $v_i \in V_i^A$, $u_j, v_j \in V_j^A$ with $F(v_i,u_j) = \alpha$, $F(v_i,v_j) = \beta$. Since $A$ is 
minimally omitted, $A-\{u_j\}, A-\{v_j\}$ are both realized in $G$. By the homogeneity of $G$, these can be embedded so that the embeddings agree on 
$A- \{ u_j, v_j \}$. Since $u_j, v_j$ are differently joined to $v_i$, they are not identified in the embedding, and so we have embedded $A$. This 
contradicts the assumption that $A$ is omitted. 
\end{proof}

To prove Theorem~\ref{nonmonic} we show that the assumption that there is $A \in O(G)$ which is non-monic leads to a contradiction. We first prove two lemmas, which determine when certain $m$-partite monics are realized or omitted in $G$, given this assumption. The set-up for all of these results is as follows: 
 
\begin{itemize}
\item[($*$)]
Let $G$ be a homogeneous $m$-generic $L$-graph with $m \ge 3$ as small as possible such that $O(G)$ contains non-monics. Let $A \in O(G)$ be non-monic 
with $|A|$ as small as possible, so then $A$ is clearly also $m$-partite. By Lemma \ref{monobip} $A$ is monochromatic with $\alpha$-coloured edges in 
all bipartite restrictions (where to simplify notation, we write $\alpha$ for a member of varying $C_{ij}$), and we assume that $|V_0^A|>1$. Let $a_0$ 
and $b_0$ be distinct members of $V_0^A$, and pick some $a_i \in V_i^A$ for each $i \in \{1, \ldots, m-1 \}$.
\end{itemize}

Now for $c \in C_{01}, ~d \in C_{12}$ we let $B^{c d}$ be the $m$-partite monic with $F(E_{01}) = c$, $F(E_{12}) = d$, and $F(E_{ij}) = \alpha$ for 
all other distinct $i,j$. Note that $B^{\alpha \alpha}$ is the $m$-partite monic monochromatic in $\alpha$, and this is realized in $G$ since it is a 
proper subgraph of $A$ which is minimally omitted. 

\begin{lem} \label{uniquemonic}
For each $d \in C_{12} - \{ \alpha \}$, there is at most one colour $c \in C_{01}$ such that $B^{c d}$ is realized in $G$. 
\end{lem}

\begin{proof}
Suppose otherwise, and let $\beta \in C_{12} - \{ \alpha \}$ and $\gamma, \delta \in C_{01}$ with $\gamma \ne \delta$ be such that $B^{\gamma \beta}, B^{\delta \beta}$ are both realized in $G$. Consider the graph $A' = A \cup \{x_1\}$ with $x_1 \in V_1$ such that $F(a_0, x_1) = \gamma$, $F(z_0, x_1) = \delta$ for each $z_0 \in V_0^A - \{a_0\}$, $F(x_1, z_2) = \beta$ for each $z_2 \in V_2^A$, $F(x_1, z_i) = \alpha$ for each $z_i \in V_i^A$ with $i \in \{3, \ldots, m-1\}$ (and the colours of edges in $A$ are unchanged, that is all other edges are coloured $\alpha$). 

Then $A' - \{ a_0, x_1 \} = A - \{a_0\} \subset A$, and $A' - \{ b_0, x_1 \} = A - \{b_0\} \subset A$ are both realized since $A$ is minimally omitted. 

Next consider $A' - \{ a_0, a_1 \}$. If this is monic, then it is a copy of $B^{\delta \beta}$ which is realized. 
Otherwise it is non-monic of size $|A' - \{ a_0, a_1 \}| = |A| - 1$. Since $A$ was chosen to be a minimally omitted non-monic of smallest possible size, if $A' - \{ a_0, a_1 \}$ is omitted, then some monic subgraph of it must be omitted. 
However, any monic subgraph which includes $x_1$ is a subgraph of $B^{\delta \beta}$ and so is realized; and any monic subgraph which does not include $x_1$ is a proper subgraph of $A$ and so is realized since $A$ is minimally omitted.
Thus $A' - \{ a_0, a_1 \}$ is realized since all monic subgraphs of it are realized. 

Similarly, consider $A' - \{ b_0, a_1 \}$. If this is monic, then it is a copy of $B^{\gamma \beta}$ which is realized. 
Otherwise it is non-monic, and again since $A$ was chosen to be a minimally omitted non-monic of smallest possible size, if $A' - \{ b_0, a_1 \}$ is omitted, then some monic subgraph of it is omitted. 
However, any monic subgraph which includes $x_1$ and $a_0$ is a subgraph of $B^{\gamma \beta}$ and so is realized; any monic subgraph which includes $x_1$ but not $a_0$ is a subgraph of $B^{\delta \beta}$ and so is realized; 
and any monic subgraph which does not include $x_1$ is a proper subgraph of $A$ and so is realized since $A$ is minimally omitted.
Thus $A' - \{ b_0, a_1 \}$ is realized since all monic subgraphs of it are realized. 

By homogeneity, $A' - \{ a_0, a_1 \}$ and $A' - \{ a_0, x_1 \}$ may be embedded in $G$ so that they agree on $A' - \{ a_0, x_1, a_1 \}$. In this embedding, $x_1, a_1$ are not identified because they are differently joined to $a_2$: observe $F(a_1, a_2) = \alpha \ne \beta = F(x_1, a_2)$. Thus $A' - \{a_0\}$ is realized in $G$. 

Similarly, by homogeneity $A' - \{ b_0, x_1 \}$ and $A' - \{ b_0, a_1 \}$ may be embedded in $G$ so that they agree on $A' - \{ b_0, x_1, a_1 \}$. Again in this embedding, $x_1, a_1$ are not identified because they are differently joined to $a_2$, so $A' - \{b_0\}$ is also realized in $G$. 

By homogeneity, $A' - \{a_0\}$ and $A' - \{b_0\}$ may be embedded in $G$ so that they agree on $A' - \{ a_0, b_0 \}$. In this embedding, $a_0, b_0$ are not identified because they are differently joined to $x_1$: observe $F(a_0, x_1) = \gamma \ne \delta = F(b_0, x_1)$. Thus $A'$ is realized in $G$. But then $A \subseteq A'$ is realized in $G$, which contradicts the assumption that it is minimally omitted.
\end{proof}

\begin{lem} \label{monicrealized}
For some $c \in C_{01} - \{ \alpha \}$, the monic $B^{c \alpha}$ is realized in $G$. 
\end{lem}

\begin{proof}
Consider $A^*$ on the same vertex set as $A$, with all edges coloured as in $A$ except that $F(a_0, a_1)$ is undefined. Now 
$A^* - \{a_0\}, ~A^* - \{a_1\} \subset A$ are realized since $A$ is minimally omitted, so by homogeneity, $A^*$ is realized. Since $A$ is omitted, 
$F(a_0, a_1) = c \ne \alpha$, and so $B^{c \alpha}$is realized.
\end{proof}

We may now prove the non-monic realization theorem.

\begin{proof}[Proof of Theorem~\ref{nonmonic}]
The set-up is as described in $(*)$ above. 

First observe that by Lemma~\ref{monicrealized}, for some $c \in C_{01} -\{ \alpha \}$, the monic $B^{c \alpha}$ is realized in $G$, say $B^{\beta \alpha}$ is realized with $\beta \ne \alpha$. 
Consider the graph $H = A \cup \{x_0\}$ with $x_0 \in V_0$, such that edges in $H - \{x_0\}$ are coloured as in $A$ except that $F(a_1, a_2)$ is undetermined; $F(x_0, z_1) = \beta$ for each $z_1 \in V_1^A$, and all other edges incident to $x_0$ are $\alpha$-coloured (that is $F(x_0, z_i) = \alpha$ for each $z_i \in V_i^A$ with $i \in \{2, \ldots, m-1\})$. We show that $H - \{a_1\}$, $H - \{a_2\}$ are both realized in $G$, and so by homogeneity $H$ (with some colour determined for the edge $a_1a_2$) should be realized---but realizing $H$ gives a contradiction, as follows. 
If $F(a_1, a_2) = \alpha$ then we have realized $A$, contrary to $A \in O(G)$. 
Otherwise, if $F(a_1, a_2) = d$ where $d \ne \alpha$, then observe that $\langle a_0, a_1, a_2, \ldots, a_{m-1} \rangle \subset H$ is a copy of $B^{\alpha d}$ and $\langle x_0, a_1, a_2, \ldots, a_{m-1} \rangle \subset H$ is a copy of $B^{\beta d}$. So for $d \in C_{12} - \{ \alpha \}$, we have realized two distinct monics $B^{\alpha d}, B^{\beta d}$ in $G$, contradicting Lemma~\ref{uniquemonic}. 

So we aim to show that both of the non-monics $H_1 := H - \{a_1\}$ and $H_2 := H - \{a_2\}$ are realized in $G$. For $i = 1,2$, first note that $H_i$ 
is a non-monic the same size as $A$, but has one fewer vertex than $A$ in part $V_i$, and one more vertex than $A$ in part $V_0$. If $|V_i^A|=1$, then 
$H_i$ is not defined on all $m$ parts, and so is not omitted by the minimal choice of $G$ and $A$. Otherwise, if $|V_i^A| > 1$, then since the 
bipartite restriction to $V_0V_1$ is not monochromatic, $H_i$ is not minimally omitted by Lemma~\ref{monobip}.

No non-monic proper subgraph of $H_i$ is minimally omitted since any such subgraph has size less than $|A|$, and $A$ was chosen to be a minimally omitted non-monic of smallest possible size. 
So if $H_i$ is not realized, then some monic subgraph of it must be omitted. However, any monic subgraph which includes $x_0$ is a subgraph of $B^{\beta \alpha}$ and so is realized; and any monic subgraph which does not include $x_0$ is a proper subgraph of $A$ and so is realized since $A$ is minimally omitted. Thus $H_i$ is realized in $G$ since all monic subgraphs of it are realized. 
\end{proof}

The following is a straightforward consequence of the non-monic realization theorem (Theorem~\ref{nonmonic}), but is worth stating explicitly since it is a very useful tool for easily determining whether a given graph is realized or omitted in a given homogeneous multipartite graph. 

\begin{cor} \label{monic}
If all monic subgraphs of $H$ are realized in a homogeneous $m$-generic graph $G$, then $H$ is realized in $G$. 
\end{cor}

\begin{proof}
By definition, a graph $H$ is omitted from $G$ if and only if some subgraph $H'$ of $H$ is minimally omitted from $G$. By Theorem~\ref{nonmonic}, such a subgraph $H'$ will be monic. So $H$ is omitted from $G$ if and only if some monic subgraph of $H$ is minimally omitted. Equivalently, $H$ is realized in $G$ if and only if no monic subgraph of $H$ is minimally omitted, that is, if and only if each monic subgraph of $H$ is realized in $G$. 
\end{proof}

\subsection{Effective characterization}

We remark that, once one has Theorem \ref{nonmonic}, even without precise information to explicitly describe each valid $O(G)$, we can say that, in an appropriate sense, the problem of determining all the countable homogeneous $m$-partite graphs in a specified (finite) language can be `effectively' solved. By this we mean that given $m$ and finite colour sets corresponding to all pairs of parts, there is an effective procedure that terminates in finitely many steps, and lists precisely the countable homogeneous $m$-partite coloured graphs with the given colour sets. 

As described in section~\ref{prelim}, in the first place, we can discount cases in which perfect matchings arise as bipartite restrictions. This is because we may appeal to Lemma~\ref{perfectm}. To begin with we choose a bipartite restriction where there is a perfect matching, and remove one of the parts. We now are considering the $(m-1)$-partite case, and we assume inductively that the possibilities for this have been effectively listed. Corresponding to each possibility which can arise, Lemma~\ref{perfectm} tells us how we can add back in the part which was removed to obtain a corresponding homogeneous $m$-partite graph. Similarly, we can discount cases with finite parts, because each bipartite restriction which involves a finite part must be complete in only one colour. 

Now looking at the case in which no perfect matching arises, and all parts are infinite, by Lemma~\ref{forbfam} it suffices to list all the 
possibilities $\cal F$ for $O(G)$, and by Theorem~\ref{nonmonic} there are only finitely many candidates for $\cal F$ which we need to examine. Given 
any $\cal F$, we can test effectively whether the class of finite $m$-partite graphs omitting all members of $\cal F$ is an amalgamation class, which 
is all we need to determine. As previously mentioned, it suffices to verify whether two-point amalgamations of the form $A = B_1 \cap B_2$ where 
$B_1 = A \cup \{x\}$ and $B_2 = A \cup \{y\}$ can be performed. We shall show that all two-point amalgamations can be performed if and only if all 
such amalgamations can be performed for $A$ of size at most $N = mg$, where $g$ is defined to be the greatest size of a colour set on a bipartite 
restriction. We can then explicitly list all the possibilities up to this bound, see if they can all be performed, and hence effectively test whether 
$Forb({\cal F})$ is an amalgamation class.

First suppose that a two-point amalgamation diagram $(A, B_1, B_2)$ where $B_1 = A \cup \{x\}$ and $B_2 = A \cup \{y\}$ is given. We shall show that $A$ has a substructure $A'$ of size at most $N$ such that the amalgamation $(A, B_1, B_2)$ can be performed if and only if the one on $A', A' \cup \{x\}, A' \cup \{y\}$ can. If the original amalgamation can be performed then we let $A' = \emptyset$. So suppose that the original amalgamation cannot be performed. Let $i$ and $j$ be such that $x \in V_i$ and $y \in V_j$. If $i = j$ then $B_1$ and $B_2$ certainly {\em can} be amalgamated over $A$ (by not joining $x$ and $y$). Hence $i \neq j$, and there is a finite colour set $C_{ij}$ associated with this bipartite restriction $V_iV_j$. Since $(A, B_1, B_2)$ cannot be amalgamated, for each $c \in C_{ij}$, the choice of $c$ as the colour for $xy$ must fail, and therefore there is some $A^c \in {\cal F}$ which is realized in $A \cup \{x, y\}$ on this choice. We assume that a fixed choice of such a subgraph $A^c$ of $A \cup \{x, y\}$ has been made. Then $A^c$ is defined on both $V_i$ and $V_j$, and $A^c - \{y\}$ and $A^c - \{x\}$ are subgraphs of $B_1$ and $B_2$ respectively. Let $A', B_1'$, and $B_2'$ be the induced subgraphs of $A \cup \{x, y\}$ on $\underset{c \in C} \bigcup A^c - \{x, y\}$, $\underset{c \in C} \bigcup A^c - \{y\}$, and $\underset{c \in C} \bigcup A^c - \{x\}$ respectively. Then $(A', B_1', B_2')$ is a two-point amalgamation problem in $Forb({\cal F})$, and $A'$ has size at most $N$. Furthermore, there is no amalgamation of $(A', B_1', B_2')$ in $Forb({\cal F})$, because whichever colour we try to assign to $xy$, we have included in $A \cup \{x, y\}$ a member of $\cal F$ which thereby gets realized. 

The rest of the paper is concerned with finding explicit conditions to classify the possible valid sets $O(G)$. Such conditions are needed to actually construct examples systematically, and can be used to see exactly why some sets are valid and others are not, unlike the above `effective' solution which gives no such information.

\section{Tripartite graphs}

In this section, we concentrate on the tripartite (3-partite) case. We are able to give a complete classification of the homogeneous 3-generic graphs. 

First let us introduce some further terminology. For a monic $A$, if $F_A(E_{ij}) = c \in C_{ij}$, then we say that $A$ \emph{covers} the colour $c$ on the restriction $V_iV_j$. Furthermore, for a set of monics ${\cal A}$, we say that ${\cal A}$ \emph{covers} $C_{ij}$, and call ${\cal A}$ a \emph{$C_{ij}$-cover set} if for each $c \in C_{ij}$ there is some $A \in {\cal A}$ which covers $c$ on $V_iV_j$. Then note that ${\cal A}$ certainly has at least $|C_{ij}|$ members, but may have more. We refer to a 3-partite monic as a \emph{triangle}.

\begin{lem} \label{3freecol}
If $G$ is a homogeneous 3-generic graph, then for each distinct $i,j \in \{0,1,2\}$, the class $O(G)$ covers at most $|C_{ij}| -1$ colours on the restriction $V_iV_j$.
\end{lem}

\begin{proof}
Otherwise, suppose that for some distinct $i,j \in \{0,1,2\}$ we have a $C_{ij}$-cover set in $O(G)$. By Theorem~\ref{nonmonic}, minimally omitted 
graphs are monic, and by Remark~\ref{A>2} each has size at least 3, so in this case each member of $O(G)$ is a triangle. Without loss of generality, 
say that for each $c \in C_{01}$ we have some triangle $A^c \in O(G)$ with its $E_{01}$-edge coloured $c$.  Suppose that for $c \in C_{01}$, the 
$E_{02}$-edge of $A^c$ is coloured $c' \in C_{02}$ and the $E_{12}$-edge of $A^c$ is coloured $c' \in C_{12}$. (Note that 
$c' \in C_{02}, ~c' \in C_{12}$ are not necessarily really the same colour---we are looking at different bipartite restrictions. Also, note that 
distinct monics in the cover set may have edges of the same colour in some bipartite restrictions, for instance we may have 
$\alpha' = \beta' \in C_{02}$ for some $\alpha \ne \beta$ in $C_{01}$. Strictly speaking, the map from $c$ to $c'$ is a function on the colours.)

Now consider $H$ on vertices $v_0 \in V_0$, $v_1 \in V_1$, and $a_2^c \in V_2$ for each $c \in C_{01}$, with $F(v_0, v_1)$ undetermined, and 
$F(v_0, a_2^c) = c' \in C_{02}$ and $F(v_1, a_2^c) = c' \in C_{12}$ for each $c \in C_{01}$. Then $H-\{v_0\}, ~H-\{v_1\}$ are both realized since they 
are bipartite and $G$ is generic on each bipartite restriction. Thus by homogeneity we can realize $H$ in $G$. But this gives a contradiction since 
for each $c \in C_{01}$, if $F(v_0, v_1) = c$ then we have realized $A^c$ on $\langle v_0, v_1, a_2^c \rangle$, but each of these is omitted.
\end{proof}

In fact, the converse of this is also true---any such class $O(G)$ defines an amalgamation class. So this gives the full classification in the 3-generic case.

\begin{thm} \label{3gen}
A 3-generic graph $G$ is homogeneous if and only if $O(G)$ is some class of triangles which for each distinct $i,j \in \{0,1,2\}$ covers at most 
$|C_{ij}| -1$ colours on the restriction $V_iV_j$.
\end{thm}

\begin{proof}
The forward direction is given by Lemma~\ref{3freecol}.

For the converse, let ${\cal F}$ be a family of triangles such that for each distinct $i,j \in \{0,1,2\}$, the family ${\cal F}$ covers at most 
$|C_{ij}| -1$ colours on the restriction $V_iV_j$. That is, for each distinct $i,j \in \{0,1,2\}$, there is some colour in $C_{ij}$ which is not 
covered by ${\cal F}$. Then it is easy to see that $Forb({\cal F})$ is an amalgamation class---2-point amalgamations can always be performed since 
there is a free colour available for the new edge. By Fra\" iss\' e's Theorem there is a countable homogeneous $3$-generic graph $G$ with 
$O(G)={\cal F}$. 
\end{proof}

Before moving on, let us look in more detail at the homogeneous 3-generic graphs given by the classification. Let $L$ be a language for coloured 
tripartite graphs, and suppose $\alpha \in C_{01} \cap C_{02} \cap C_{12}$. Consider the family ${\cal F}$ of all triangle $L$-graphs without any 
$\alpha$-coloured edges. Observe that $|{\cal F}| = \underset{i \ne j} \prod (|C_{ij}|-1)$. By Theorem~\ref{3gen}, for each 
${\cal F'} \subseteq {\cal F}$, the class $Forb({\cal F'})$ is an amalgamation class, and there is a homogeneous 3-generic $L$-graph $G$ with 
$O(G) = {\cal F'}$. Furthermore, up to colour-isomorphism, these are the only possible homogeneous 3-generic $L$-graphs. Note that this list of 
homogeneous $L$-graphs reduces when considering the graphs up to colour-isomorphism (for example, in a fixed language $L$, the homogeneous 3-generic 
$L$-graphs which omit just one triangle are all colour-isomorphic).

As a specific example, suppose that in the language $L$, for each distinct $i, j \in \{0,1,2\}$ we have $|C_{ij}| = 3$ and $\alpha \in C_{ij}$, and 
${\cal F'}$ is any subset of the family of 8 triangles with no $\alpha$-coloured edges. Then $Forb({\cal F'})$ is an amalgamation class, and there is 
a homogeneous 3-generic $L$-graph $G$ with $O(G) = {\cal F'}$. 

\section{Omission sets in $m$-generic graphs}

As we have seen, the homogeneous $m$-generic graphs are characterized by the monics that they omit. For tripartite graphs, the characterization of possible families of minimally omitted monics was relatively simple. However for $m \ge 4$,  things become more complicated. In the ordinary non-coloured graph case, the key configurations of omitted monics to understand were \emph{omission quartets}, consisting of certain \emph{differing} triangles (see \cite{jenkinson2}). In this section we look at the relevant generalizations of these notions for the coloured case. 

First we introduce some further basic notation and terminology. Let $G$ be an $m$-partite graph defined on parts $V_0, V_1, \ldots, V_{m-1}$. For $J \subseteq \{0, \ldots, m-1\}$, we call a subgraph $G'$ of $G$ defined on the parts $\underset{i \in J} \bigcup V_i$ a \emph{$J$-subgraph}. For example, a subgraph $G'$ defined on the parts $V_1, V_3, V_4$ of a 5-partite graph, is called a \emph{$134$-subgraph}. Furthermore, if $G'$ is monic, then we call it a \emph{$134$-monic} (or in this case a \emph{$134$-triangle}). If monics $A,B$ have the same colour $E_{ij}$-edge, then we say that $A,B$ \emph{agree} on this edge-type. 

Recall that a $C_{ij}$-cover set is a set of monics which covers all colours in $C_{ij}$. For distinct $i,j,k,l \in \{ 0, \ldots, m-1 \}$, we define a \emph{$C_{ij}^{kl}$-omission set} to be a set of $ijk$-triangles and $ijl$-triangles which form a $C_{ij}$-cover set, such that the triangles all agree on their $E_{ik}, E_{il}, E_{jk}, E_{jl}$-edges (that is, all of the triangles have the same colours on the bipartite restrictions $V_iV_k, V_iV_l, V_jV_k, V_jV_l$). If $S_{ij}$ is a $C_{ij}^{kl}$-omission set such that all of the triangles have $c_{ik}$-coloured $E_{ik}$-edges, $c_{il}$-coloured $E_{il}$-edges, $c_{jk}$-coloured $E_{jk}$-edges, and $c_{jl}$-coloured $E_{jl}$-edges, then we define the \emph{code} for $S_{ij}$ to be the tuple $(i,j,k,l; c_{ik},c_{il},c_{jk},c_{jl})$. That is, each omission set is `coded' by the four parts on which it is defined, and the four colours of the agreeing edge-types.

If a $C_{kl}^{ij}$-omission set $S_{kl}$ has the `same' code as $S_{ij}$, that is, it is defined on the same four parts, and has the same colours on the same four agreeing edge-types (so $S_{kl}$ is a set of $ikl$-triangles and $jkl$-triangles which all agree with $S_{ij}$ on their $E_{ik}, E_{il}, E_{jk}, E_{jl}$-edges), then we say that $S_{ij}, S_{kl}$ are \emph{corresponding} omission sets. These corresponding pairs of omission sets made up of triangles generalize the notion of omission quartets from the ordinary non-coloured graph case. An omission quartet in an ordinary $m$-generic graph is precisely a pair of corresponding omission sets for which each relevant colour set has size two (and so there are exactly four members of each omission quartet).

Consider a code $\Omega = (i,j,k,l; c_{ik},c_{il},c_{jk},c_{jl})$. For $c \in C_{ij}$, let $T_k^c$ be the $ijk$-triangle agreeing with $\Omega$ (that is, with a $c_{ik}$-coloured $E_{ik}$-edge and a $c_{jk}$-coloured $E_{jk}$-edge) and with $F(E_{ij})=c$, and similarly let $T_l^c$ be the $ijl$-triangle agreeing with $\Omega$ and with $F(E_{ij})=c$. If $S_{ij}$ is a $C_{ij}^{kl}$-omission set with code $\Omega$, then for each $c \in C_{ij}$ either $T_k^c$ or $T_l^c$ lies in $S_{ij}$. Note that for general $C_{ij}$-cover sets, $S_{ij}$ may have more than $|C_{ij}|$ members---here this means that for some $c \in C_{ij}$ both $T_k^c$ and $T_l^c$ may lie in $S_{ij}$. 

\begin{rem} \label{same3}
By Lemmas \ref{restriction} and \ref{3freecol}, there is no $C_{ij}$-cover set defined on just 3 parts. So a $C_{ij}^{kl}$-omission set with code $\Omega$ never contains only $ijk$-triangles $\{T_k^c: c \in C_{ij}\}$ or only $ijl$-triangles $\{T_l^c: c \in C_{ij}\}$ agreeing with $\Omega$. In particular, a $C_{ij}^{kl}$-omission set always contains both $ijk$-triangles and $ijl$-triangles. 
\end{rem}

\begin{lem} \label{corros}
Let $G$ be a homogeneous $m$-generic graph. If there is a $C_{ij}^{kl}$-omission set in $O(G)$, then there is a corresponding $C_{kl}^{ij}$-omission set in $O(G)$. 
\end{lem}

\begin{proof}
Let $S_{ij}$ be a $C_{ij}^{kl}$-omission set in $O(G)$ with code $\Omega = (i,j,k,l; c_{ik},c_{il},c_{jk},c_{jl})$. Then for each $c \in C_{ij}$, at least one of the triangles $T_k^c$ or $T_l^c$ agreeing with $\Omega$ (as defined above) lies in $S_{ij}$. 

Similarly, for each $d \in C_{kl}$, let $T_i^d, T_j^d$ be the $ikl$-triangle and $jkl$-triangle respectively agreeing with $\Omega$ and with $d$-coloured $E_{kl}$-edges. Suppose that for some $d \in C_{kl}$, both $T_i^d$ and $T_j^d$ are realized in $G$. Then by the homogeneity of $G$, we can embed them to agree on their agreeing $E_{kl}$-edges. Say that we have $v_i, v_j, v_k, v_l \in G$ such that $\langle v_i, v_k, v_l \rangle$ is a copy of $T_i^d$ and $\langle v_j, v_k, v_l \rangle$ is a copy of $T_j^d$. Now consider the edge $v_iv_j$: if $F(v_i,v_j) = c \in C_{ij}$, then $\langle v_i, v_j, v_k \rangle$ is a copy of $T_k^c$ and $\langle v_i, v_j, v_l \rangle$ is a copy of $T_l^c$. But for each $c \in C_{ij}$ at least one of these triangles lies in $S_{ij}$, so is omitted, which gives a contradiction. 

Thus for each $d \in C_{kl}$, at least one of $T_i^d, T_j^d$ is omitted. Call such a triangle $T^d$ and let $S_{kl} = \{ T^d : d \in C_{kl} \}$. By 
Remark~\ref{same3}, $S_{kl}$ must have some $ikl$-triangles and some $jkl$-triangles. Thus $S_{kl} \subseteq O(G)$ is a $C_{kl}^{ij}$-omission set 
with code $\Omega$ corresponding to $S_{ij}$. 
\end{proof}

\section{Quadripartite graphs}

In this section, we concentrate on the quadripartite (4-partite) case, and we see how by understanding omission sets we are able to give a complete classification of the homogeneous 4-generic graphs.

First we introduce another relevant piece of terminology. If ${\cal A}$ is a $C_{ij}$-cover set, we say that $S_{ij}$ is a 
$C_{ij}^{kl}$-omission set \emph{based on} ${\cal A}$ if there is a pair of monics $A^c, A^d \in {\cal A}$ with the colours on 
the $E_{ik}, E_{jk}$-edges of members of $S_{ij}$ agreeing with those in $A^c$, and the colours on the $E_{il}, E_{jl}$-edges 
of members of $S_{ij}$ agreeing with those in $A^d$. That is, $S_{ij}$ with code 
$\Omega = (i,j,k,l; c_{ik}, c_{il}, c_{jk}, c_{jl})$ is based on ${\cal A}$ if there are $A^c \in {\cal A}$ such that 
$F_{A^c}(E_{ij})=c, ~F_{A^c}(E_{ik})=c_{ik}, ~F_{A^c}(E_{jk})=c_{jk}$, and $A^d \in {\cal A}$ such that 
$F_{A^d}(E_{ij})=d, ~F_{A^d}(E_{il})=c_{il}, ~F_{A^d}(E_{jl})=c_{jl}$. We shall also say that $S_{ij}$ is \emph{based on} 
$A^c, A^d$.

\begin{lem} \label{4partite}
Let $G$ be a homogeneous 4-generic graph, and let ${\cal A}$ be a $C_{ij}$-cover set in $O(G)$ for some distinct 
$i,j \in \{0,1,2,3\}$. Then there is a $C_{ij}^{kl}$-omission set in $O(G)$ based on ${\cal A}$, where 
$\{k,l\} = \{0,1,2,3\} - \{i,j\}$. Furthermore, ${\cal A}$ must have some $ijk$-triangles and some $ijl$-triangles.
\end{lem}

\begin{proof}
Without loss of generality, we may assume that ${\cal A}$ is a $C_{01}$-cover set, $|{\cal A}| = |C_{01}|$, and that 
${\cal A} = \{ A^c : c \in C_{01} \}$ where for each $c \in C_{01}$,  $A^c$ is a monic with $c$-coloured $E_{01}$-edge. Let 
$\chi_1 = \{ c \in C_{01} : A^c \textrm{ is a 0123-monic} \}$, 
$\chi_2 = \{ c \in C_{01} : A^c \textrm{ is a 012-triangle} \}$, 
$\chi_3 = \{ c \in C_{01} : A^c \textrm{ is a 013-triangle} \}$, so $C_{01} = \chi_1 \cup \chi_2 \cup \chi_3$. Note that 
$\chi_1$ may be empty, and by Remark~\ref{same3} we cannot have either $\chi_1 = \chi_2 = \emptyset$ or 
$\chi_1 = \chi_3 = \emptyset$. 

For each $c \in C_{01}$, we shall refer to the colours of all edges of $A^c$ (other than the $c$-coloured $E_{01}$-edge) as 
$c'$. So if $(i, j) \neq (0, 1)$ and $A^c$ is defined on $V_i$ and $V_j$, then on the bipartite restriction $V_iV_j$, 
$c' := F_{A^c}(E_{ij}) \in C_{ij}$. Thus, as in the proof of Lemma~\ref{3freecol}, for each bipartite restriction $V_iV_j$ we 
think of the map taking $c$ to $c'  \in C_{ij}$ as a function on the set of colours $C_{01}$ (except note that now the map may 
only be partially defined for some bipartite restrictions).

Now construct a 4-partite graph $H$ with vertices $v_0 \in V_0, v_1 \in V_1$, $a_2^c \in V_2$ for each $c \in \chi_1 \cup \chi_2$, and $a_3^c \in V_3$ for each $c \in \chi_1 \cup \chi_3$. We colour the edges of $H$ as follows: for each $c \in \chi_1 \cup \chi_2$, let $F(v_0, a_2^c) = c' \in C_{02}$ and $F(v_1, a_2^c) = c' \in C_{12}$ (so these edges agree with the $E_{02}, E_{12}$-edges of $A^c$); similarly for each $c \in \chi_1 \cup \chi_3$, let $F(v_0, a_3^c) = c' \in C_{03}$ and $F(v_1, a_3^c) = c' \in C_{13}$; and for each $c \in \chi_1$, let $F(a_2^c, a_3^c) = c' \in C_{23}$. The colour of $v_0v_1$ is undefined, and it remains to decide the colours of the other $E_{23}$-edges.

We claim that for some distinct pair $\alpha \in \chi_1 \cup \chi_2, ~\beta \in \chi_1 \cup \chi_3$, there is a $C_{23}^{01}$-omission set $S_{23}$ in 
$O(G)$ with colours on the $E_{02}, E_{12}, E_{03}, E_{13}$-edges as in $\langle v_0, v_1, a_2^\alpha, a_3^\beta \rangle \subseteq H$---that is, with 
code $\Omega^{\alpha \beta}=(0,1,2,3; \alpha' \in C_{02}, ~\beta' \in C_{03}, ~\alpha' \in C_{12}, ~\beta' \in C_{13})$. If this is not the case, then 
for each such pair $\alpha, \beta$ there must be some colour $c^{\alpha \beta} \in C_{23}$ such that the 023-triangle with $\alpha'$-coloured 
$E_{02}$-edge, $\beta'$-coloured $E_{03}$-edge, and $c^{\alpha \beta}$-coloured $E_{23}$-edge, and the 123-triangle with $\alpha'$-coloured 
$E_{12}$-edge, $\beta'$-coloured $E_{13}$-edge, and $c^{\alpha \beta}$-coloured $E_{23}$-edge, are both realized in $G$. Now for each such distinct 
pair $\alpha, \beta$, let $F(a_2^\alpha, a_3^\beta) = c^{\alpha \beta}$ in $H$. We aim to show that $H - \{v_0\}$ and $H - \{v_1\}$ are both realized 
in $G$. Consider any (triangle) monic subgraph $\langle x, a_2^\gamma, a_3^\delta \rangle$ of $H - \{v_0\}$ or $H - \{v_1\}$, where 
$x \in \{v_0, v_1 \}$. If $\gamma \ne \delta$, then this is a triangle we have just mentioned, and we assumed that the colour $c^{\gamma\delta}$ of 
the edge $a_2^\gamma a_3^\delta$ could be chosen so that these triangles are realized in $G$. Otherwise, if $\gamma=\delta$, then this is a proper 
subgraph of the 4-monic $A^\gamma$ in $O(G)$, hence the triangle is realized. Thus by Corollary \ref{monic}, the graphs $H - \{v_0\}$ and 
$H - \{v_1\}$ are both realized in $G$, and so by the homogeneity of $G$, the graph $H$ with some defined edge $v_0 v_1$ is realized in $G$. But for 
each $c \in C_{01}$ if $F(v_0, v_1) = c$ under this embedding, then we have also realized $A^c$ in $G$ (on 
$\langle v_0, v_1, v_i^c: i \in I^c - \{0,1\} \rangle$), which gives a contradiction since each of these is omitted. 

Note that by Remark~\ref{same3}, $S_{23}$ must have some 023-triangles and some 123-triangles. By Lemma~\ref{corros} there is also a corresponding $C_{01}^{23}$-omission set $S_{01}$ in $O(G)$ with code $\Omega^{\alpha \beta}$. The $E_{02}, E_{12}$-edges of the 012-triangles in $S_{01}$ agree with $A^\alpha$, and the $E_{03}, E_{13}$-edges of the 013-triangles in $S_{01}$ agree with $A^\beta$, so $S_{01}$ is a $C_{01}^{23}$-omission set based on ${\cal A}$ (in particular based on $A^\alpha, A^\beta$). 

Finally, suppose that the original $C_{01}$-cover set ${\cal A}$ has no 013-triangles, that is, $\chi_3 = \emptyset$, and aim for a contradiction. Note that we may follow the same construction as above to find a $C_{23}^{01}$-omission set $S_{23}$, and corresponding $C_{01}^{23}$-omission set $S_{01}$ based on $A^\alpha, A^\beta \in {\cal A}$ in $O(G)$. Since $\chi_3 = \emptyset$, we know $\beta \in \chi_1$, that is, $A^\beta$ is a 0123-monic. As $A^\beta$ is minimally omitted, its 013-subgraph must be realized in $G$. That is, the 013-triangle with $\beta$-coloured $E_{01}$-edge, $\beta'$-coloured $E_{03}$-edge, and $\beta'$-coloured $E_{13}$-edge is not included in the $C_{01}^{23}$-omission set $S_{01}$. Because $S_{01}$ is a $C_{01}^{23}$-omission set, the 012-triangle $B^\beta$ with $\beta$-coloured $E_{01}$-edge, $\alpha'$-coloured $E_{02}$-edge, and $\alpha'$-coloured $E_{12}$-edge must lie in $S_{01}$. But then we could replace $A^\beta$ in ${\cal A}$ by $B^\beta$ to get a new $C_{01}$-cover set in $O(G)$ with no 013-triangles and one fewer 4-partite monic. By repeating this procedure, we could find a cover set in $O(G)$ with no 4-partite monics at all. This would be a cover set consisting of only 012-triangles, which contradicts Remark~\ref{same3}. Hence each $C_{01}$-cover set must have some $012$-triangles and some $013$-triangles. 
\end{proof}

These conditions are in fact enough to characterize all of the homogeneous coloured 4-generic graphs. 

\begin{thm} \label{4generic}
Let $L$ be a language for coloured quadripartite graphs, and let ${\cal F}$ be a family of pairwise non-embeddable monic $L$-graphs. Then there is a 
(unique) countable homogeneous 4-generic $L$-graph $G$ with $Age(G) = Forb({\cal F})$ if and only if ${\cal F}$ satisfies the following: 
\begin{enumerate}
\item[(i)] if ${\cal A} \subseteq {\cal F}$ is a $C_{ij}^{kl}$-omission set, then there is a corresponding $C_{kl}^{ij}$-omission set ${\cal B}$ in 
${\cal F}$;
\item[(ii)]  if ${\cal A} \subseteq {\cal F}$ is a $C_{ij}$-cover set, then there is a $C_{ij}^{kl}$-omission set in ${\cal F}$ based on ${\cal A}$ 
(so in particular, ${\cal A}$ must have some $ijk$-triangles and some $ijl$-triangles).
\end{enumerate}
\end{thm}

\begin{proof}
The forward direction is given by Lemmas~\ref{corros} and \ref{4partite}.

For the converse, we must show that if ${\cal F}$ satisfies conditions (i) and (ii), then $Forb({\cal F})$ is an amalgamation class. 
So let $A, B_1, B_2 \in Forb({\cal F})$ be such that $A$ embeds in $B_1$ and $B_2$. We may assume that $B_1 = A \cup \{x\}, ~B_2 = A \cup \{y\}$. Now 
we just need to decide the colour of the edge $xy$ to form the amalgam $C = A \cup \{x,y\}$ so that no member of ${\cal F}$ is realized. 

If $x,y$ are in the same part $V_i$, then no decision is required. Vertices $x,y$ are not joined, and certainly we shall not have realized any member 
of ${\cal F}$ since these graphs are monic (if $C$ realizes $D \in {\cal F}$, then $D$ was already realized in $B_1$ or $B_2$, but 
$B_1, B_2 \in Forb({\cal F})$ so this does not happen).

So without loss of generality, suppose $x \in V_0, ~y \in V_1$. If there is some colour $c \in C_{01}$ such that $C$ with $F(x,y) = c$ does not 
realize any members of ${\cal F}$, then we may colour the edge $xy$ by $c$ and we have finished. Otherwise, for each $c \in C_{01}$, the graph $C$ 
with $F(x,y) = c$ realizes some $A^c \in {\cal F}$ with $c$-coloured $E_{01}$-edge. We aim to show that this will never happen. First note that the 
set ${\cal A} := \{ A^c : c \in C_{01} \}$ is a $C_{01}$-cover set in ${\cal F}$. By (ii), there is a $C_{01}^{23}$-omission set $S_{01}$ in 
${\cal F}$ based on ${\cal A}$. Thus there are some $v_2 \in V_2^C$, $v_3 \in V_3^C$ such that the $C_{01}^{23}$-omission set 
$S_{01} \subseteq {\cal F}$ agrees with $\langle x, y, v_2, v_3 \rangle$ on $E_{02}, E_{03}, E_{12}, E_{13}$-edges. But then by (i), ${\cal F}$ also 
contains a corresponding $C_{23}^{01}$-omission set $S_{23}$. Now consider the edge $v_2v_3$ in $C$, say $F(v_2, v_3) = d \in C_{23}$. At least one of 
$\langle x, v_2, v_3 \rangle$ or $\langle y, v_2, v_3 \rangle$ with $F(v_2, v_3) = d$ lies in $S_{23} \subseteq {\cal F}$. But $x, v_2, v_3 \in B_1$ 
and $y, v_2, v_3 \in B_2$, so one of $B_1, B_2$ realizes a forbidden monic, which contradicts our initial assumptions. 
\end{proof}

At this point let us consider some examples of homogeneous 4-generic graphs given by this classification. Let $L$ be a fixed language for quadripartite graphs. We give some examples of families of finite $L$-graphs ${\cal F}$ for which $Forb({\cal F})$ is an amalgamation class. Recall that we only really want to consider families of pairwise non-embeddable monics to guarantee that all members of ${\cal F}$ are in fact minimally omitted by the homogeneous graph $G$ with $Age(G) = Forb({\cal F})$, so that $O(G) = {\cal F}$.

To describe these examples, it is useful to introduce some further notation for coloured monic $m$-partite graphs. A $J$-monic $A$ (defined on parts 
$J \subseteq \{0, \ldots, m-1 \}$) will be denoted by $[J; \overline{\chi}]$ where $\overline{\chi}$ is an ordered tuple listing the colours of all 
the edges of $A$, that is $\overline{\chi} = (c_{ij} \in C_{ij}: F_A(E_{ij})=c_{ij} \textrm{ for } i,j \in J \textrm{ with } i<j)$ where the 
$c_{ij}$s are ordered by the lexicographic ordering on the pairs $(i,j) \in J^2$. For example, if $A$ is a $0235$-monic such that for each 
$i,j \in \{0,2,3,5\}$ with $i<j$, $F_A(E_{ij})=c_{ij} \in C_{ij}$, then $\overline{\chi} = (c_{02}, c_{03}, c_{05}, c_{23}, c_{25}, c_{35})$. 
Using this notation the $123$-triangle with $\alpha$-coloured $E_{12}$-edge, $\alpha$-coloured $E_{13}$-edge, and $\beta$-coloured $E_{23}$-edge is 
denoted by $[\{1,2,3 \}; (\alpha, \alpha, \beta)]$, which is further abbreviated to $[123; \alpha \alpha \beta]$. 

In each of the following examples, consider the language $L$ such that $C_{ij}=\{ \alpha, \beta, \gamma \}$ for each distinct $i,j \in \{0,1,2,3\}$. 

\begin{eg}
Examples without any cover sets:

If ${\cal F} = \{ [0123; \alpha \alpha \alpha \alpha \alpha \alpha], [012; \beta \beta \beta], [023; \alpha \beta \beta] \}$, then $Forb({\cal F})$ is an amalgamation class, and there is a unique homogeneous 4-generic $L$-graph $G$ with $O(G) = {\cal F}$. 
This is because if ${\cal F}$ is any set of monics which does not contain a cover set, then $Forb({\cal F})$ is an amalgamation class. 

Similarly, if ${\cal F}$ is the set of all 32 triangle $L$-graphs with no $\gamma$-coloured edges, and ${\cal F}' \subseteq {\cal F}$, then $Forb({\cal F}')$ is an amalgamation class, and there is a unique homogeneous 4-generic $L$-graph $G$ with $O(G) = {\cal F}'$. 
\end{eg}

\begin{eg}
The most basic type of example which includes a cover set:

Let $A_1 = [012; \alpha \alpha \alpha]$, $A_2 = [012; \beta \alpha \alpha]$, $A_3 = [013; \gamma \alpha \alpha]$, $A_4 = [023; \alpha \alpha \alpha]$, $A_5 = [123; \alpha \alpha \beta]$, $A_6 = [123; \alpha \alpha \gamma]$, and let ${\cal F}_1 = \{ A_i : 1 \le i \le 6 \}$. Then $Forb({\cal F}_1)$ is an amalgamation class, and there is a unique homogeneous 4-generic $L$-graph $G_1$ with $O(G_1) = {\cal F}_1$. 
Observe that $\{ A_1, A_2, A_3\}$ is a $C_{01}^{23}$-omission set, and $\{ A_4, A_5, A_6\}$ is a corresponding $C_{23}^{01}$-omission set, but there are no other cover sets. 
\end{eg}

\begin{eg}
Let $A_7 = [013; \beta \alpha \alpha]$, $A_8 = [023; \alpha \alpha \beta]$, and $A^* = [0123; \beta \beta \beta \beta \beta \beta]$. For each ${\cal A} \subseteq \{A_7, A_8, A^* \}$, $Forb({\cal F}_1 \cup {\cal A})$ is an amalgamation class. 
\end{eg}

\begin{eg}
A `maximal' example:

$A_1 = [012; \alpha \alpha \alpha]$, $B_1 = [012; \beta \alpha \alpha]$, $B_2 = [012; \alpha \beta \alpha]$, $B_3 = [012; \alpha \alpha \beta]$, 
$A_2 = [013; \alpha \alpha \alpha]$, $C_1 = [013; \gamma \alpha \alpha]$, $B_4 = [013; \alpha \beta \alpha]$, $B_5 = [013; \alpha \alpha \beta]$, 
$A_3 = [023; \alpha \alpha \alpha]$, $C_2 = [023; \gamma \alpha \alpha]$, $C_3 = [023; \alpha \gamma \alpha]$, $B_6 = [023; \alpha \alpha \beta]$, 
$A_4 = [123; \alpha \alpha \alpha]$, $C_4 = [123; \gamma \alpha \alpha]$, $C_5 = [123; \alpha \gamma \alpha]$, $C_6 = [123; \alpha \alpha \gamma]$. 
Let ${\cal F}_2 = \{A_i : 1 \le i \le 4 \} \cup \{B_i : 1 \le i \le 6 \} \cup \{C_i : 1 \le i \le 6 \}$. Then $Forb({\cal F}_2)$ is an amalgamation class, giving homogeneous 4-generic $G_2$ with $O(G_2) = {\cal F}_2$.

Observe that ${\cal F}_2$ contains $C_{ij}$-cover sets for each distinct $i,j \in \{0,1,2,3\}$. It contains the following `overlapping' omission sets: \\
$\{ A_1, B_1, A_2, C_1\}$ is a $C_{01}^{23}$-omission set, and $\{ A_3, B_6, A_4, C_6\}$ is a corresponding $C_{23}^{01}$-omission set; 
$\{ A_1, B_2, A_3, C_2\}$ is a $C_{02}^{13}$-omission set, and $\{ A_2, B_5, A_4, C_5\}$ is a corresponding $C_{13}^{02}$-omission set; 
$\{ A_1, B_3, A_4, C_4\}$ is a $C_{12}^{03}$-omission set, and $\{ A_2, B_4, A_3, C_3\}$ is a corresponding $C_{03}^{12}$-omission set. 

In fact ${\cal F}_2$ is `maximal', that is, $Forb({\cal F}_2)$ is an amalgamation class, but there is no ${\cal F'} \supset {\cal F}_2$ in the language $L$ such that $Forb({\cal F'})$ is an amalgamation class.
\end{eg}

\begin{eg}
Another `maximal' example:

$A_1 = [012; \alpha \alpha \alpha]$, 
$A_2 = [012; \beta \alpha \alpha]$, 
$A_3 = [013; \alpha \alpha \alpha]$, 
$A_4 = [013; \gamma \alpha \alpha]$, 
$A_5 = [023; \alpha \alpha \alpha]$, 
$A_6 = [023; \alpha \alpha \beta]$, 
$A_7 = [123; \alpha \alpha \alpha]$, 
$A_8 = [123; \alpha \alpha \gamma]$,

$B_1 = [012; \alpha \alpha \gamma]$, 
$B_2 = [012; \beta \alpha \gamma]$, 
$B_3 = [013; \alpha \gamma \alpha]$, 
$B_4 = [013; \gamma \gamma \alpha]$, 
$B_5 = [023; \alpha \gamma \alpha]$, 
$B_6 = [023; \alpha \gamma \beta]$, 
$B_7 = [123; \gamma \alpha \alpha]$, 
$B_8 = [123; \gamma \alpha \gamma]$,

$C_1 = [012; \alpha \gamma \alpha]$, 
$C_2 = [012; \beta \gamma \alpha]$, 
$C_3 = [013; \alpha \alpha \gamma]$, 
$C_4 = [013; \gamma \alpha \gamma]$, 
$C_5 = [023; \gamma \alpha \alpha]$, 
$C_6 = [023; \gamma \alpha \beta]$, 
$C_7 = [123; \alpha \gamma \alpha]$, 
$C_8 = [123; \alpha \gamma \gamma]$,

$D_1 = [012; \alpha \gamma \gamma]$, 
$D_2 = [012; \beta \gamma \gamma]$, 
$D_3 = [013; \alpha \gamma \gamma]$, 
$D_4 = [013; \gamma \gamma \gamma]$, 
$D_5 = [023; \gamma \gamma \alpha]$, 
$D_6 = [023; \gamma \gamma \beta]$, 
$D_7 = [123; \gamma \gamma \alpha]$, 
$D_8 = [123; \gamma \gamma \gamma]$.

Let ${\cal F}_3 = \{A_i : 1 \le i \le 8 \} \cup \{B_i : 1 \le i \le 8 \} \cup \{C_i : 1 \le i \le 8 \} \cup \{D_i : 1 \le i \le 8 \}$. Then $Forb({\cal F}_3)$ is an amalgamation class (and as in the previous example, it is maximal), giving homogeneous 4-generic $G_3$ with $O(G_3) = {\cal F}_3$.

Observe that ${\cal F}_3$ only contains $C_{01}$-cover sets and $C_{23}$-cover sets. It
contains many overlapping pairs of $C_{01}^{23}$-omission sets and corresponding $C_{23}^{01}$-omission sets. 
For example, if $X \in \{ A,B,C,D \}$, then $\{X_1, X_2, X_3, X_4\}$ and $\{X_5, X_6, X_7, X_8\}$ are corresponding omission sets. 
Also, $\{A_1, A_2, B_3, B_4\}$ and $\{B_5, B_6, A_7, A_8\}$ are corresponding omission sets; as are $\{B_1, B_2, C_3, C_4\}$ and $\{A_5, A_6, D_7, D_8\}$. There are 16 such pairs in all, since in this example there are 16 possible omission set codes $(0,1,2,3; c_{02}, c_{03}, c_{12}, c_{13})$ for $c_{02}, c_{03}, c_{12}, c_{13} \in \{\alpha, \gamma\}$.
\end{eg}

\section{Non-complication for $m$-generic graphs}

We work towards showing that as in the monochromatic case in \cite{jenkinson2}, even for multipartite graphs on more than four parts, omission sets 
are the key to determining amalgamation classes. That is, to characterize the homogeneous graphs, things do not get more complicated than 
corresponding omission sets---to verify the amalgamation property for a given class $Forb({\cal F})$ we do not have to worry about any more 
complicated configurations of omitted monics. Formally, we aim to show that Lemma~\ref{4partite} generalizes to $m$-partite graphs with $m \ge 4$. The 
`non-complication theorem' is the following:

\begin{thm} \label{mgen}
Let $G$ be a homogeneous $m$-generic graph, and let ${\cal A}$ be a $C_{ij}$-cover set in $O(G)$. Then there is a $C_{ij}^{kl}$-omission set in $O(G)$ based on ${\cal A}$ (for some distinct $k,l \in \{0, \ldots, m-1 \} - \{i,j\}$).
\end{thm}

Furthermore, as in the quadripartite case, such an omission set must in fact be based on triangles from the cover set. Before embarking on the proof of the main non-complication theorem, we shall prove this strengthening result which restricts the kinds of cover sets that can arise.

\begin{lem} \label{triangleomset}
If $G$ is a homogeneous $m$-generic graph, ${\cal A}$ is a $C_{ij}$-cover set in $O(G)$, and there is a $C_{ij}^{kl}$-omission set in $O(G)$ based on ${\cal A}$, then there is a $C_{ij}^{kl}$-omission set in $O(G)$ based on triangles from ${\cal A}$ (so in particular ${\cal A}$ must contain some $ijk$-triangles and some $ijl$-triangles for some distinct $k,l \in \{0, \ldots, m-1 \} - \{i,j\}$).
\end{lem}

Our goal is to formulate explicit conditions which will determine whether a given family ${\cal F}$ of monic $L$-graphs equals $O(G)$ for some homogeneous $L$-graph $G$. Lemma~\ref{triangleomset} says that if we can find a cover set in ${\cal F}$ which does not contain two types of triangles, then ${\cal F}$ does not have the form $O(G)$. 

Let us further develop the terminology and notation which will be used in the proofs of these results. 
If ${\cal A}$ is a $C_{ij}$-cover set in $O(G)$, we let $\Omega_{ij}({\cal A}) := \{ (i,j,k,l; c_{ik},c_{il},c_{jk},c_{jl}) :$ there are 
$A^c \in {\cal A}$ such that $F_{A^c}(E_{ij})=c, ~F_{A^c}(E_{ik})=c_{ik}, ~F_{A^c}(E_{jk})=c_{jk}$, and $A^d \in {\cal A}$ such that 
$F_{A^d}(E_{ij})=d, ~F_{A^d}(E_{il})=c_{il}, ~F_{A^d}(E_{jl})=c_{jl} \}$. Thus $\Omega_{ij}({\cal A})$ denotes the set of all `possible' codes for 
$C_{ij}^{kl}$-omission sets based on ${\cal A}$. That is, if there is a $C_{ij}^{kl}$-omission set in $O(G)$ based on ${\cal A}$, then it must have 
some code in $\Omega_{ij}({\cal A})$. 
But of course the converse does not necessarily hold. In particular, note that in the definition of $\Omega_{ij}({\cal A})$, the monics $A^c, A^d$ from which we build a code are not necessarily distinct (so such codes built from a single monic are contained in $\Omega_{ij}({\cal A})$); however if there is a $C_{ij}^{kl}$-omission set in $O(G)$ with code $\Omega \in \Omega_{ij}({\cal A})$, then it is clear that there is no minimally omitted monic which agrees with the code on all four of the relevant edge-types, as we now show.

\begin{lem} \label{agreeing}
Let $G$ be a homogeneous $m$-generic graph, and let $S_{ij} \subseteq O(G)$ be a $C_{ij}^{kl}$-omission set with code $\Omega$. 
Then no $A \in O(G)$ is defined on all four parts $V_i, V_j, V_k, V_l$ and agrees with $\Omega$ on the $E_{ik}, E_{il}, E_{jk}, E_{jl}$-edges. 
\end{lem}

\begin{proof}
Let $\Omega = (i,j,k,l; c_{ik},c_{il},c_{jk},c_{jl})$, and suppose that $A \in O(G)$ agrees with $\Omega$ on the relevant edge-types, and 
$F_A(E_{ij})=\alpha$. As previously, for $c \in C_{ij}$, define $T_k^c, T_l^c$ to be the $ijk$-triangle and $ijl$-triangle respectively agreeing with 
$\Omega$ and with $c$-coloured $E_{ij}$-edges. Then one of $T_k^\alpha$ or $T_l^\alpha$ is included in the omission set $S_{ij}$, so at least one of 
them is minimally omitted from $G$. But both of these triangles are proper subgraphs of $A$, so $A$ itself cannot be minimally omitted, a 
contradiction. 
\end{proof}

If $\Omega \in \Omega_{ij}({\cal A})$ then we shall say that the code $\Omega$ is \emph{based on} ${\cal A}$. Then Theorem~\ref{mgen} says that for 
each $C_{ij}$-cover set ${\cal A}$ in $O(G)$, there is some code $\Omega$ based on ${\cal A}$ such that there is a $C_{ij}^{kl}$-omission set in 
$O(G)$ with code $\Omega$. 

We further extend the `based on' terminology to introduce an ordering on $C_{ij}$-cover sets as follows. For ${\cal A}, {\cal B}$ both $C_{ij}$-cover
sets in $O(G)$, we write ${\cal A} \le {\cal B}$, and say that ${\cal A}$ is \emph{based on} ${\cal B}$, if 
$\Omega_{ij}({\cal A}) \subseteq \Omega_{ij}({\cal B})$. That is, ${\cal A} \le {\cal B}$ if each code based on ${\cal A}$ is also a code based on 
${\cal B}$. Furthermore write ${\cal A} < {\cal B}$ (and say that ${\cal A}$ is \emph{strictly} based on ${\cal B}$) if 
$\Omega_{ij}({\cal A}) \subset \Omega_{ij}({\cal B})$ (roughly speaking, ${\cal B}$ has strictly more codes based on it than ${\cal A}$).

To prove Lemma~\ref{triangleomset}, we suppose for a contradiction that there is a cover set ${\cal A}$ in $O(G)$ such that there is a corresponding omission set in $O(G)$ based on ${\cal A}$, but none based on triangles from ${\cal A}$. 
By previous results, we know that certain `bad' cover sets cannot exist in a valid $O(G)$. 
The notion `bad' is informal (it will not be explicitly defined) but by it we understand some non-valid cover set, whose existence contradicts a previous result. 
For instance, note that by Lemma~\ref{agreeing}, if ${\cal A} \subseteq O(G)$ is a cover set, then a corresponding omission set based on ${\cal A}$ cannot be based on just one member of ${\cal A}$. So if ${\cal A}$ is a $C_{ij}$-cover set in $O(G)$ such that for each code in $\Omega_{ij}({\cal A})$ there is some monic in ${\cal A}$ which agrees with the whole of the code, then ${\cal A}$ is bad. Similarly, a cover set defined on only three parts is bad since it contradicts Remark~\ref{same3}; and a cover set defined on only four parts which contains only one kind of triangle is bad since it contradicts Lemma~\ref{4partite}. 
Under the given conditions, the non-existence of such `bad' cover sets can then be used to obtain the desired contradiction. 

\begin{proof}[Proof of Lemma \ref{triangleomset}]
Suppose for a contradiction that ${\cal A}$ is a $C_{01}$-cover set in $O(G)$, such that there is a $C_{01}$-omission set in $O(G)$ based on 
${\cal A}$, but none based on triangles from ${\cal A}$. We may suppose that $|{\cal A}|=|C_{01}|$, and ${\cal A} = \{ A^c : c \in C_{01} \}$, where 
each $A^c$ is a monic with $c$-coloured $E_{01}$-edge and all other edges $c'$-coloured. Let $S_{01}$ be a $C_{01}^{kl}$-omission set (for distinct 
$k,l \in \{2,\ldots,m-1\}$) based on $A^\alpha, A^\beta \in {\cal A}$, with code 
$\Omega = (0,1,k,l; \alpha' \in C_{0k}, ~\beta' \in C_{0l}, ~\alpha' \in C_{1k}, ~\beta' \in C_{1l})$. So $A^\alpha$ is an $I^\alpha$-monic with 
$0, 1, k \in I^\alpha$, and $A^\beta$ is an $I^\beta$-monic with $0, 1, l \in I^\beta$. Since $S_{01}$ is not based on triangles, without loss of 
generality assume that there are no $01k$-triangles in ${\cal A}$ agreeing with $\Omega$, so in particular $|I^\alpha| >3$.

As previously, for $c \in C_{01}$, let $T_k^c, T_l^c$ be the $01k$-triangle and $01l$-triangle respectively agreeing with $\Omega$ and with 
$c$-coloured $E_{01}$-edges. Then $S_{01} = {\cal T}_k \cup {\cal T}_l$ where ${\cal T}_k = \{ T_k^c : c \in \chi_k \}$ and \mbox{${\cal T}_l = \{ T_l^c : c \in \chi_l \}$} for some $\chi_k, \chi_l \subset C_{01}$ with $\chi_k \cup \chi_l = C_{01}$. 

We may assume that $S_{01}$ is optimal in the following sense: $|S_{01}|=|C_{01}|$ (so $\chi_k \cap \chi_l = \emptyset$), and we maximize 
$|{\cal T}_l|$ (equivalently maximize $|\chi_l|$). That is, if for some $c \in C_{01}$ both $T_k^c$ and $T_l^c$ are minimally omitted, then we choose 
to put $T_l^c$ and not $T_k^c$ in $S_{01}$. 
Note that since $A^\alpha$ is minimally omitted (but is not a triangle), $T_k^\alpha \subset A^\alpha$ is realized in $G$, so we must have $T_l^\alpha$ in $S_{01}$ (that is, $\alpha \in \chi_l$). 

If $|\chi_l|=1$, then $\chi_l = \{\alpha\}$ and $\chi_k = C_{01} - \{\alpha\}$ (so note that $\beta \notin \chi_l$, so $A^\beta$ is not a triangle). 
Consider ${\cal T}_k \cup \{A^\alpha\}$, which is a $C_{01}$-cover set in $O(G)$. This is a bad cover set: it consists of one non-triangle monic and all the other monics are $01k$-triangles, but this contradicts Lemma~\ref{4partite}. In fact since $A^\alpha$ and all of the $01k$-triangles agree on their $E_{0k}, E_{1k}$-edges (each has $\alpha'$-coloured $E_{0k}$-edges and $E_{1k}$-edges), $A^\alpha$ must agree with any code based on this cover set, so this contradicts Lemma~\ref{agreeing}.

Thus we may assume that $|\chi_l|>1$. Now consider ${\cal A'} = {\cal T}_l \cup \{ A^c : c \in \chi_k \}$. This is a $C_{01}$-cover set in $O(G)$, and 
clearly $\Omega_{01}({\cal A'}) \subseteq \Omega_{01}({\cal A})$ so ${\cal A'} \le {\cal A}$. 
We show that actually ${\cal A'} < {\cal A}$, since there is no monic in ${\cal A'}$ agreeing with $A^\alpha \in {\cal A}$ on the $E_{0k}, E_{1k}$-edges, and so $\Omega \notin \Omega_{01}({\cal A'})$. 
Otherwise, suppose we have such a monic. Then since this monic is not a $01l$-triangle, it must be $A^\gamma$ for some $\gamma \in \chi_k$. 
Clearly $A^\gamma$ is not a triangle, since by our initial assumptions there is no $01k$-triangle in ${\cal A}$ agreeing with $\Omega$. 
But then $T_k^\gamma \subset A^\gamma$ is realized in $G$, which contradicts the fact that $\gamma \in \chi_k$. 
So indeed, ${\cal A'} < {\cal A}$. 

Now we run through the argument again, this time with an omission set $S'_{01}$ with code $\Omega'$ based on ${\cal A'}$ (but not based on triangles). By induction, we find a chain of $C_{01}$-cover sets ${\cal A} > {\cal A'} > {\cal A''} > \ldots$ such that there is no corresponding omission set based on triangles from any of them. But since we are working in a finite language, $\Omega_{01}({\cal A})$ is finite, so this chain cannot go on indefinitely, and we have a contradiction. 
\end{proof}

\subsection{Proving the non-complication theorem} \label{sectionm}

We now turn to the proof of the full non-complication theorem (Theorem~\ref{mgen}). The idea of the proof is to assume that there is some homogeneous 
$m$-generic graph $G$ with a cover set in $O(G)$ (which we may assume covers $C_{01}$) but no corresponding omission set based on that cover set, and 
aim for a contradiction. So assuming such homogeneous graphs exist, we consider a minimal case---that is, with $m$ as small as possible such that 
$O(G)$ contains a $C_{01}$-cover set which does not have an omission set based on it, and ${\cal A}$ is a `minimal' such $C_{01}$-cover set. First 
note that by Lemma~\ref{4partite}, certainly $m \ge 5$; and ${\cal A}$ is defined on all $m$ parts of $G$. To be `minimal', we require that ${\cal A}$ 
satisfies a number of further conditions. Thus justification for assuming these conditions will be given in the results that follow. At this stage we 
just describe what these `minimality' conditions are.

Firstly, we may assume that ${\cal A}$ contains the minimum number of monics for a $C_{01}$-cover set, that is $|{\cal A}| = |C_{01}|$. Next 
${\cal A}$ should be minimal with respect to the `based on' ordering $<$ (that is, there is no $C_{01}$-cover set ${\cal B}$ in $O(G)$ with 
$\Omega_{01}({\cal B}) \subset \Omega_{01}({\cal A})$). In fact, we may assume that ${\cal A}$ is a special kind of such a cover set, which we shall call 
\emph{good}. A $C_{01}$-cover set is \emph{good} if there is a particular colour $\alpha \in C_{01}$ (called the \emph{key colour}) such that for each monic in the cover set, if we recolour its
$E_{01}$-edge by $\alpha$, then the resulting monic is either a subgraph of $A^\alpha$ (the monic in the cover set with a $\alpha$-coloured $E_{01}$-edge, called the \emph{key monic}) or is 
realized in $G$.

Furthermore, we may assume that ${\cal A}$ is a good cover set such that the key monic is either defined on at least four parts, or is defined on exactly three parts and there are no other monics in ${\cal A}$ defined on all of the remaining parts. We call such a good cover set a \emph{star} cover set.

Finally, if possible, we choose ${\cal A}$ which is $(i,j)$-free for some distinct $i,j \in \{2, \ldots, m-1\}$; where a cover set ${\cal A}$ is 
\emph{$(i,j)$-free} if no members of ${\cal A}$ are defined on both $V_i$ and $V_j$. 

\subsubsection{`Minimal' cover sets: minimum number of members}

Before embarking on the proof of the theorem, we fully describe our set-up, introduce the necessary new notions that will be used, and cover some preliminary results involving these. 

First note that each $C_{01}$-cover set in $O(G)$ is a subset of the set of all monics in $O(G)$ which are defined on $V_0$ and $V_1$. We wish to consider $C_{01}$-cover sets with exactly $|C_{01}|$ members (that is, $C_{01}$-cover sets that contain exactly one monic with $c$-coloured $E_{01}$-edge for each $c \in C_{01}$). So let $r := |C_{01}|$, $C_{01} = \{ 0,1, \ldots, r-1 \}$, and let $\Lambda$ be the set of all $C_{01}$-cover sets in $O(G)$ of size $r$.
It is straightforward to see that we may assume that `minimal' ${\cal A}$ is a $C_{01}$-cover set in $\Lambda$, since by the following, if there is a $C_{01}$-cover set in $O(G)$ without an omission set based on it, then there is such a $C_{01}$-cover set of size $r$.

\begin{lem} \label{lambda}
If for each ${\cal A} \in \Lambda$, there is an $C_{01}$-omission set in $O(G)$ based on ${\cal A}$, then in fact this holds for each $C_{01}$-cover set in $O(G)$.
\end{lem}

\begin{proof}
Let ${\cal B}$ be a $C_{01}$-cover set in $O(G)$. Consider any subset ${\cal A}$ of ${\cal B}$ of size $r$ such that ${\cal A}$ is a $C_{01}$-cover set. Then ${\cal A} \in \Lambda$, and so by the hypothesis, there is a $C_{01}$-omission set ${\cal S}$ based on ${\cal A}$, say based on monics $A^c, A^d \in {\cal A}$. Now $A^c, A^d \in {\cal B}$ since ${\cal A} \subseteq {\cal B}$, so ${\cal S}$ is also based on ${\cal B}$.
\end{proof}

We use the following standard notation for a cover set ${\cal A} \in \Lambda$: let ${\cal A} = \{ A^c : c \in C_{01} \}$, where for each 
$c \in C_{01}$, $A^c$ is an $I^c$-monic with $\{0,1\} \subset I^c \subseteq \{ 0,1,\ldots, m-1\}$, with vertices $v^c_i$ for each $i \in I^c$, with 
$c$-coloured $E_{01}$-edge and all other edges $c'$-coloured. 

\subsubsection{`Minimal' cover sets: minimal in `based on' ordering}

Next we stated that `minimal' ${\cal A}$ should be minimal with respect to the `based on' ordering $<$. Roughly speaking, this means that the monics 
in the cover set have as much agreement as possible. In fact, we claimed that we may assume that this `minimal' ${\cal A}$ is a good cover 
set---which tells us that certain other monics are realized. In the following, we formalize these ideas (which involves quite a bit of ground work), with the ultimate goal of proving that we can indeed assume that `minimal' ${\cal A}$ is good by showing that given any cover set we can always find a good cover set based on it.

First we set up some further terminology and notation which will be useful. So consider some ${\cal A} \in \Lambda$. The set $\Omega_{01}({\cal A})$ of all codes based on ${\cal A}$ tells us something about the amount of agreement between the monics in the cover set; but we wish to further quantify this. Since all members of ${\cal A}$ differ on their $E_{01}$-edges, clearly none is a subgraph of any other. However we can consider `almost subgraphs'; that is, when one graph is a subgraph of another except for their differing $E_{01}$-edges. 
For $c,d \in C_{01}$, define $A^{d/c}$ to be an isomorphic copy of $A^d$ except with $c$-coloured $E_{01}$-edge (rather than $d$-coloured). 
Now if $A^c$ is a subgraph of $A^{d/c}$, then we say that $A^c$ is an \emph{off-$E_{01}$-subgraph} of $A^d$, and write $A^c \le_{01} A^d$.
Furthermore, write $A^c <_{01} A^d$ if $A^c$ is a proper off-$E_{01}$-subgraph of $A^d$.

The relation $\le_{01}$ is a quasi-order on each $C_{01}$-cover set ${\cal A} \in \Lambda$. But note that it need not be a partial order because 
$\le_{01}$ may not be antisymmetric: $A^c \le_{01} A^d$ and $A^d \le_{01} A^c$  means that $A^c, A^d$ are isomorphic except for their differently 
coloured $E_{01}$-edges. In this case we shall say that $A^c, A^d$ are \emph{off-$E_{01}$-isomorphic}, and write $A^c =_{01} A^d$. For a cover set 
${\cal A} \in \Lambda$, we are particularly interested in the maximal elements of the structure $({\cal A}, \le_{01})$; in the following these may 
simply be referred to as \emph{maximal elements of ${\cal A}$}.

Recall the ordering on $C_{ij}$-cover sets ${\cal A}, {\cal B}$ determined by the codes based on the cover sets: ${\cal B} \le {\cal A}$ if and only 
if $\Omega_{ij}({\cal B}) \subseteq \Omega_{ij}({\cal A})$ (and then we say that ${\cal B}$ is based on ${\cal A}$), and let 
$\Lambda_{\cal A} := \{ {\cal B} \in \Lambda : {\cal B} \le {\cal A} \}$, the set of all cover sets in $\Lambda$ based on ${\cal A} \in \Lambda$, and
 for each ${\cal B} \in \Lambda_{\cal A}$, define $t_{\cal B}$ to be the number of maximal elements of $({\cal B}, \le_{01})$, up to 
off-$E_{01}$-isomorphism. Let $t({\cal A}) := \underset{{\cal B} \in \Lambda_{\cal A}} \min t_{\cal B}$; and we say that 
${\cal B} \in \Lambda_{\cal A}$ is \emph{${\cal A}$-minimal} if $t_{\cal B} = t({\cal A})$; that is, ${\cal B}$ has the least possible number of maximal elements for a cover set based on ${\cal A}$.
We first observe that $t_{\cal A} \ge 2$ for each ${\cal A} \in \Lambda$; that is, all $C_{01}$-cover sets in $\Lambda$ have at least two non-off-$E_{01}$-isomorphic maximal elements. Then clearly $t({\cal A}) \ge 2$ for each ${\cal A} \in \Lambda$ too. (In fact, by the same argument this is true for all $C_{01}$-cover sets, not just those in $\Lambda$, but we would need to extend the terminology used and change the notation a little in the proof.)

\begin{lem} \label{mincoverset}
If $G$ is a homogeneous $m$-generic graph which has a $C_{01}$-cover set in $O(G)$, then $t_{\cal A} \ge 2$ for each ${\cal A} \in \Lambda$. 
\end{lem}

\begin{proof}
Otherwise, suppose that ${\cal A} \in \Lambda$ with $t_{\cal A} = 1$, and we aim for a contradiction. Then ${\cal A} \in \Lambda$ is a 
$C_{01}$-cover set with a unique (up to off-$E_{01}$-isomorphism) maximal element $A^\alpha$, so for each $c \in C_{01}$, $A^c \le_{01} A^\alpha$. 

The graphs $A^\alpha - \{v^\alpha_0\}, ~A^\alpha - \{v^\alpha_1\}, ~A^\alpha - \{v^\alpha_0, v^\alpha_1 \}$ are all realized in $G$, since they are 
all proper subgraphs of $A^\alpha$ (which is minimally omitted). So we can amalgamate $A^\alpha - \{v^\alpha_0\}$ and 
$A^\alpha - \{v^\alpha_1\}$ over $A^\alpha - \{v^\alpha_0, v^\alpha_1 \}$ to get a graph realized in $G$; call this amalgam $H$. The only new edge of 
$H$ is $v^\alpha_0 v^\alpha_1 \in E_{01}$, and we let $F(v^\alpha_0 v^\alpha_1) = \beta \in C_{01}$. But now $A^\beta \le_{01} H$, since 
$A^\beta \le_{01} A^\alpha$ and $H =_{01} A^\alpha$. Since the $E_{01}$-edge of $H$ is $\beta$-coloured, in fact $A^\beta \subseteq H$, so we have 
realized $A^\beta$ in $G$, contradicting the fact that it is minimally omitted. 
\end{proof}

We introduce a new order on cover sets in $\Lambda$, using the corresponding  $\le_{01}$ orders and maximal elements. For distinct 
${\cal A}, {\cal B} \in \Lambda$, we write ${\cal A} \preceq {\cal B}$ if each maximal element of $({\cal A}, \le_{01})$ is an off-$E_{01}$-subgraph 
of a maximal element of $({\cal B}, \le_{01})$. Note that the order $\le$ (determined by the sets of codes based on the cover sets) extends the order $\preceq$. 

\begin{lem} \label{orders}
Let ${\cal A}, {\cal B}$ be $C_{01}$-cover sets in $\Lambda $.
If ${\cal A} \preceq {\cal B}$, then ${\cal A} \le {\cal B}$.
\end{lem}

\begin{proof}
Suppose ${\cal A} \preceq {\cal B}$, and consider $\Omega = (0,1,k,l; c_{0k}, c_{0l}, c_{1k}, c_{1l}) \in \Omega_{01}({\cal A})$. Say $\Omega$ is based on the monics $A^c, A^d \in {\cal A}$. That is, the $E_{0k}, E_{1k}$-edges in $A^c$ have colours $c_{0k}, c_{1k}$ respectively, and the $E_{0l}, E_{1l}$-edges in $A^d$ have colours $c_{0l}, c_{1l}$ respectively. In fact there are maximal elements $A^\gamma, A^\delta$ of $({\cal A}, \le_{01})$ which $\Omega$ is based on (for instance, say $A^\gamma, A^\delta$ are maximal elements with $A^c \le_{01} A^\gamma, ~A^d \le_{01} A^\delta$). 
Then since ${\cal A} \preceq {\cal B}$, there are maximal elements $B^\alpha, B^\beta$ of $({\cal B}, \le_{01})$ with $A^\gamma \le_{01} B^\alpha, ~A^\delta \le_{01} B^\beta$. Now the $E_{0k}, E_{1k}$-edges in $B^\alpha$ have the same colours as the corresponding edges in $A^\gamma$ (that is, colours $c_{0k}, c_{1k}$ respectively); and the $E_{0l}, E_{1l}$-edges in $B^\beta$ have the same colours as the corresponding edges in $A^\delta$ (that is, colours $c_{0l}, c_{1l}$ respectively). So $\Omega \in \Omega_{01}({\cal B})$, and hence $\Omega_{01}({\cal A}) \subseteq \Omega_{01}({\cal B})$.
\end{proof}

However we may note that the orders are different, since there can exist cover sets ${\cal A}, {\cal B} \in \Lambda$ with $\Omega_{01}({\cal A}) = \Omega_{01}({\cal B})$, but ${\cal A} \npreceq {\cal B}$ and ${\cal B} \npreceq {\cal A}$.
Also observe that there can exist ${\cal A}, {\cal B} \in \Lambda$ with ${\cal A} \prec {\cal B}$ (that is, ${\cal A} \preceq {\cal B}$ but ${\cal B} \npreceq {\cal A}$, so there is some maximal monic in ${\cal B}$ which is not an off-$E_{01}$-subgraph of any monic in ${\cal A}$) and $\Omega_{01}({\cal A}) = \Omega_{01}({\cal B})$, so the strict relations are not necessarily preserved in the extension from $\preceq$ to $\le$.

\subsubsection{`Minimal' cover sets: good cover sets}

We now formally describe good cover sets, and show how to find a good cover set based on any cover set. Recall the definition, which can now be written as follows: a cover set ${\cal A} = \{ A^c : c \in C_{01} \}$ in $\Lambda$ is \emph{good} if there is some colour $\alpha \in C_{01}$ such that for all $c \in C_{01} - \{ \alpha \}$, either $A^c \le_{01} A^\alpha$, or $A^{c/\alpha}$ is realized in $G$. In this case we call $\alpha$ a \emph{key colour}, and $A^\alpha$ a \emph{key monic}.

\begin{rem} \label{key}
If ${\cal A} \in \Lambda$ is good and $A^\alpha$ is a key monic, then $A^\alpha$ is a maximal element of ${\cal A}$. Otherwise suppose that there is some $\beta \in C_{01}$ such that $A^\alpha <_{01} A^\beta$; then certainly $A^\beta \not\le_{01} A^\alpha$, and $A^\alpha \subset A^{\beta/\alpha}$ so $A^{\beta/\alpha}$ is not realized in $G$; so $A^\alpha$ cannot be a key monic.

By contrast, note that if ${\cal A} \in \Lambda$ is not good, and $A^\alpha$ is a maximal element of ${\cal A}$, then there must be some $\beta \in C_{01}$ with $A^\beta \not\le_{01} A^\alpha$ such that $A^{\beta/\alpha}$ is not realized in $G$, since $\alpha$ is not a key colour. 
Supposing that $A^{\beta/\alpha}$ is not realized in $G$, some monic subgraph of it is minimally omitted. Note that all subgraphs not defined on both $V_0$ and $V_1$ are realized in $G$ since these are proper subgraphs of $A^{\beta}$ which is minimally omitted. So there is some minimally omitted monic subgraph of $A^{\beta/\alpha}$ which is defined on $V_0$ and $V_1$; that is, there is a monic $D^{\alpha} \subseteq A^{\beta/\alpha}$ with $\alpha$-coloured $E_{01}$-edge with $D^{\alpha} \in O(G)$.
This observation plays an important role in the following.
\end{rem}

We aim towards proving that if ${\cal A}$ is a $C_{01}$-cover set in $\Lambda$ (that is, if $C_{01}$-cover sets exist in $O(G)$), then there is a good 
$C_{01}$-cover set ${\cal B} \in \Lambda$ with ${\cal B} \preceq {\cal A}$ (so in particular, by Lemma~\ref{orders}, ${\cal B}$ is based on 
${\cal A}$). We write an algorithm which, given a homogeneous $m$-generic graph $G$ and a $C_{01}$-cover set ${\cal B}_0 \in \Lambda$ in $O(G)$, 
either finds a good $C_{01}$-cover set ${\cal B} \in \Lambda$ with ${\cal B} \preceq {\cal B}_0$, or finds a $C_{01}$-cover set ${\cal D} \in \Lambda$ 
with ${\cal D} \prec {\cal B}_0$ and with one fewer maximal element (up to off-$E_{01}$-isomorphism) than ${\cal B}_0$ (that is, $t_{\cal D} = t_{{\cal B}_0} - 1$). 

We fully describe, formally define, and verify, this \emph{good cover set search algorithm} (or GCSSA) in the next subsection (\ref{sectionGCSSA}), but for now let us just 
briefly explain what the algorithm does, and then show how it allows us to complete the proof of the non-complication theorem (Theorem~\ref{mgen}). The input for the algorithm is 
any cover set ${\cal B}_0 \in \Lambda$. At each step $i$, we consider a particular maximal member $B^{\gamma_i}$ of the cover set 
${\cal B}_i \in \Lambda$. If $B^{\gamma_i}$ is a key monic, then ${\cal B}_i$ is good, and we stop. Otherwise, by Remark~\ref{key}, we find another 
minimally omitted monic $D^{\gamma_i} \in O(G)$ such that $D^{\gamma_i} \subseteq B^{\delta_i / \gamma_i}$ for some suitable $\delta_i \in C_{01}$, 
$\delta_i \ne \gamma_i$, with $B^{\delta_i} \in {\cal B}_i$, and replace $B^{\gamma_i}$ by $D^{\gamma_i}$ in the next step to obtain 
${\cal B}_{i+1} \in \Lambda$ (so note that ${\cal B}_{i+1} \preceq {\cal B}_i$). 
By the careful choice of the monic $B^{\gamma_i}$ at each step, if the algorithm does not find a good cover set, then we shall show 
that when it stops (after a finite number of steps) the cover set ${\cal D}$ at that step has one fewer maximal element (up to 
off-$E_{01}$-isomorphism) than ${\cal B}_0$.

Firstly, we use the GCSSA to prove that if ${\cal A}$ is a $C_{01}$-cover set in $\Lambda$, then there is a good $C_{01}$-cover set ${\cal B}$ in $\Lambda$ with ${\cal B} \preceq {\cal A}$ (and so by Lemma~\ref{orders}, ${\cal B}$ is based on ${\cal A}$).

\begin{lem} \label{goodcoverset}
If $G$ is a homogeneous $m$-generic graph which has a $C_{01}$-cover set ${\cal A}$ in $O(G)$, then there is a good $C_{01}$-cover set in $\Lambda$ 
based on ${\cal A}$. 
\end{lem}

\begin{proof}
We may assume that ${\cal A} \in \Lambda$. Otherwise consider some subset of the cover set ${\cal A}' \subset {\cal A}$ such that ${\cal A}' \in \Lambda$. Any good $C_{01}$-cover set based on ${\cal A}'$ is also based on ${\cal A}$.

Suppose that there are no good cover sets in $\Lambda$ based on ${\cal A} \in \Lambda$, and aim for a contradiction.
Consider the set $\Lambda_{\cal A}$ of all cover sets in $\Lambda$ based on ${\cal A}$. Then by the assumption, there are no good cover sets in $\Lambda_{\cal A}$.
Let ${\cal B}_0 \in \Lambda_{\cal A}$ be an ${\cal A}$-minimal cover set (that is, it has the least possible number of maximal elements (up to off-$E_{01}$-isomorphism) for a cover set based on $\cal A$), and apply the GCSSA. At each step, the cover set considered is a member of $\Lambda_{\cal A}$ (since each is based on the previous one), and it is not good. So the algorithm stops by finding a new cover set ${\cal D} \in \Lambda_{\cal A}$ with ${\cal D} \preceq {\cal B}_0$ such that $t_{\cal D} = t_{{\cal B}_0} - 1$. But now ${\cal D}$ has strictly fewer maximal elements (up to off-$E_{01}$-isomorphism) than ${\cal B}_0$, which clearly contradicts the fact that ${\cal B}_0$ was ${\cal A}$-minimal. 
\end{proof}

Furthermore, we use the GCSSA to show that for any $C_{01}$-cover set ${\cal A}$, we may find such a good $C_{01}$-cover set ${\cal B}$ based on ${\cal A}$ which is a \emph{star} $C_{01}$-cover set. That is, ${\cal B}$ either has a key monic $B^\alpha$ defined on at least 4 parts, or if the key monic $B^\alpha$ is defined on exactly 3 parts (say $B^\alpha$ is a $012$-monic) then there are no other monics in ${\cal B}$ defined on all of the remaining parts (in this instance, $V_3, \ldots, V_{m-1}$; that is, there are no $0134\ldots(m-1)$-monics nor $01\ldots(m-1)$-monics in ${\cal B}$).

\begin{lem} \label{starcoverset}
If $G$ is a homogeneous $m$-generic graph which has a $C_{01}$-cover set ${\cal A}$ in $O(G)$, then there is a star $C_{01}$-cover set in $\Lambda$ based on ${\cal A}$. 
\end{lem}

The proof of this lemma involves knowing more about how the algorithm works, so we leave the proof until later (see subsection \ref{sectionGCSSA}).
However using this result, we now prove Theorem~\ref{mgen}, by showing that if there exist $C_{01}$-cover sets which do not have corresponding 
$C_{01}$-omission sets based on them, then we can always find a minimal such $C_{01}$-cover set which is a star cover set; 
but we also show that if such a star cover set exists which does not have any corresponding 
$C_{01}$-omission sets based on it, then it is possible to construct a graph realized in $G$ which must realize some member of the star cover set, giving the required contradiction.

\subsubsection{The proof of the non-complication theorem}

\begin{proof}[Proof of Theorem~\ref{mgen}]
Assuming for a contradiction that there exist homogeneous $n$-generic graphs with a cover set (in the set of all minimally omitted monics), but no corresponding omission set based on that cover set, we may consider a minimal case.
So suppose $G$ is a homogeneous $m$-generic graph with $m \ge 1$ as small as possible such that $O(G)$ contains a cover set (which we may assume covers $C_{01}$), but no corresponding omission set based on that cover set; and aim for a contradiction. First observe that such a cover set must be defined on all $m$ parts of $G$, and by Lemma~\ref{4partite}, certainly $m \ge 5$.

Now, by the preceding results, we may assume that ${\cal A}$ is a particular kind of `minimal' such $C_{01}$-cover set. That is,
${\cal A} \in \Lambda$ (so $|{\cal A}| = |C_{01}|$) by Lemma~\ref{lambda};
${\cal A}$ is minimal with respect to the `based on' ordering $<$ (that is, there is no $C_{01}$-cover set ${\cal B}$ in $O(G)$ with $\Omega_{01}({\cal B}) \subset \Omega_{01}({\cal A})$);
and ${\cal A}$ is a star cover set by Lemmas~\ref{goodcoverset} and \ref{starcoverset}. 
Also, if possible, find such a cover set ${\cal A}$ which is $(i,j)$-free for some distinct $i,j \in \{2, \ldots, m-1\}$.
However, in the rest of the proof, we show that the assumption that there are no corresponding $C_{01}$-omission sets in $O(G)$ based on this minimal star cover set ${\cal A}$ leads to a contradiction, by showing how to construct a graph realized in $G$ which must realize some member of the star cover set.

So let ${\cal A} \in \Lambda$ be a star $C_{01}$-cover set. Without loss of generality, we may assume that the key colour is $0$, and so ${\cal A}$ has key monic $A^0$. 
We show how to construct certain graphs which are closely related to $A^0$, but which are realized in $G$. 

To construct the graphs, we `work over' a fixed part on which $A^0$ is defined (other than $V_0, V_1$). 
If $A^0$ is a triangle, then we may assume that $I^0 = \{0,1,2\}$, and so we work over $V_2$.
Otherwise, if $A^0$ is defined on at least 4 parts, then we consider two cases. If ${\cal A}$ is $(i,j)$-free for some $i,j \in \{2, \ldots, m-1\}$, then note that $|I^0 - \{0,1,i,j\}| \ge 1$, so we may assume that $2 \in I^0 - \{0,1,i,j\}$ (that is, $i \ne 2 \ne j$ and $2 \in I^0$); so again we may work over $V_2$.
Finally, if ${\cal A}$ is not $(i,j)$-free for any choice of $i,j \in \{2, \ldots, m-1\}$ (and so by the choice of ${\cal A}$, this is the case for all such good cover sets without corresponding omission sets), then we just assume that $2 \in I^0$ and work over $V_2$ (but end up finding a contradiction to this choice of ${\cal A}$). 
So in all cases we work over $V_2$, having carefully chosen this part.

\begin{claim} 
There are colours $\alpha \in C_{02} - \{0'\}$, $\beta \in C_{12} - \{0'\}$ such that the three monics $A^0(02;\alpha)$, $A^0(12;\beta)$, and $T^{0 \alpha \beta}$ are all realized in $G$: 
where $A^0(02;\alpha)$ is a copy of $A^0$ except with $\alpha$-coloured $E_{02}$-edge; $A^0(12;\beta)$ is a copy of $A^0$ except with $\beta$-coloured $E_{12}$-edge; and $T^{0 \alpha \beta}$ is a triangle (a $012$-monic) with $0$-coloured $E_{01}$-edge, $\alpha$-coloured $E_{02}$-edge, and $\beta$-coloured $E_{12}$-edge. 
\end{claim}

\begin{proof}[Proof of claim]
First note that since $A^0$ is minimally omitted, $A^0 - \{ v^0_0 \}$ and $A^0 - \{ v^0_2 \}$ are both realized in $G$. We may amalgamate these over 
$A^0 - \{ v^0_0, v^0_2 \}$ to find some monic $A^0(02;\alpha)$ with $\alpha \in C_{02}$ which is realized in $G$. Note that $\alpha \ne 0'$ since 
$A^0(02;0') = A^0$ which is minimally omitted, so $\alpha \in C_{02} - \{0'\}$.

Now consider the following two graphs: 
$D_1$ has vertices $a_0, b_0 \in V_0$, and $v_i \in V_i$ for each $i \in I^0 - \{0,2\}$, with $F(a_0 v_1) = F(b_0 v_1) = 0$ and all other edges $0'$-coloured; and $D_2$ has vertices $a_0, b_0 \in V_0$, and $v_i \in V_i$ for each $i \in I^0 - \{0,1\}$, with $F(a_0 v_2) = \alpha$ and all other edges $0'$-coloured (including $F(b_0 v_2) = 0'$).

The two maximal monic subgraphs $\langle a_0, v_1, v_3, \ldots, v_{|I^0|-1} \rangle$ and $\langle b_0, v_1, v_3, \ldots, v_{|I^0|-1} \rangle$ of $D_1$ 
are both copies of $A^0 - \{v^0_2\}$, and are realized since $A^0$ is minimally omitted. Thus by Corollary~\ref{monic}, $D_1$ is realized in $G$.

Meanwhile, $D_2$ also has two maximal monic subgraphs. Firstly $\langle a_0, v_2, \ldots, v_{|I^0|-1} \rangle \subset D_2$ is a copy of $A^0(02;\alpha) - \{v_1\}$, which is realized since $\alpha$ was chosen so that $A^0(02;\alpha)$ is realized. And secondly 
$\langle b_0, v_2, \ldots, v_{|I^0|-1} \rangle$ is a copy of $A^0 - \{v^0_1\}$, which is realized since $A^0$ is minimally omitted. Thus by 
Corollary~\ref{monic}, $D_2$ is realized in $G$.

We may now amalgamate $D_1$ and $D_2$ over $D_1 - \{v_1\} = D_2 - \{v_2\}$ to obtain a graph $H$ realized in $G$. The only new edge of $H$ is 
$v_1 v_2$, and we let $F_H(v_1 v_2) = \beta \in C_{12}$. Then $\langle b_0, v_1, \ldots, v_{|I^0|-1} \rangle \subset H$ is a copy of $A^0(12;\beta)$ 
realized in $G$, and since $A^0(12;0') = A^0$, which is minimally omitted, certainly $\beta \ne 0'$ so $\beta \in C_{12} - \{0'\}$.

We also note that $\langle a_0, v_1, v_2 \rangle \subset H$ is a copy of $T^{0 \alpha \beta}$. So now we have found 
$\alpha \in C_{02} - \{0'\}$, $\beta \in C_{12} - \{0'\}$ such that the monics $A^0(02;\alpha), ~A^0(12;\beta)$, and $T^{0 \alpha \beta}$ are all 
realized in $G$, as required.
\end{proof}

Now let $A^*$ be an $I^0$-graph with vertices $v_i \in V_i$ for $i \in \{0,1\}$, $u^0_2, w^0_2 \in V_2$, and $v^0_i$ for each $i \in I^0 - \{0,1,2\}$, such that $F(v_0, v_1) = 0$, $F(v_0, u^0_2) = \alpha$, $F(v_1, w^0_2) = \beta$ and all other edges are $0'$-coloured. 
The two maximal monic subgraphs of $A^*$ are $\langle v_0, v_1, u^0_2, v^0_3, \ldots, v^0_{|I^0|-1} \rangle$ which is a copy of $A^0(02;\alpha)$, and $\langle v_0, v_1, w^0_2, v^0_3, \ldots, v^0_{|I^0|-1} \rangle$ which is a copy of $A^0(12;\beta)$, which are both realized by the choice of $\alpha, \beta$, and so by Corollary~\ref{monic}, $A^*$ is realized in $G$.

Next let $t := t_{\cal A} = t({\cal A}) \le r$, and let ${\cal A}^M$ be a set of $t$ pairwise non-off-$E_{01}$-isomorphic maximal elements of ${\cal A}$, such that $A^0 \in {\cal A}^M$. 
Let $M := \{ c \in C_{01} : A^c \in {\cal A}^M \}$; without loss of generality we may assume that $M = \{ 0, 1, \ldots, t-1 \}$.
For each $c \in M - \{0\} = \{ 1, \ldots, t-1 \}$, observe that $A^{c/0}$ is realized in $G$, since ${\cal A}$ is a good cover set and $0$ is a key colour. 
We now wish to show that for each $c \in \{ 1, \ldots, t-1 \}$, there are colours $\delta^{0c}_{ij}  \in C_{ij}$ for each $i \in I^0 - \{0,1,2\}$, $j \in I^c - \{0,1\}$ with $i<j$; 
$\delta^{c0}_{ij} \in C_{ij}$ for each $i \in I^c - \{0,1\}$, $j \in I^0 - \{0,1,2\}$ with $i<j$;  
and $\gamma^{0c}_{2j} \in C_{2j}$ for each $j \in I^c - \{0,1,2\}$, 
such that we can realize the following graph $B^{c*}$. 

Let $I^{c*}:= I^c \cup I^0 \subseteq \{ 0,1,\ldots, m-1\}$, and let $B^{c*}$ be an $I^{c*}$-graph with vertices $v_0 \in V_0$, $v_1 \in V_1$, $u^0_2, w^0_2 \in V_2$, $v^0_i \in V_i$ for each $i \in I^0 - \{ 0,1,2 \}$, and $v^c_i \in V_i$ for each $i \in I^c - \{ 0,1 \}$. 
Edge colours agree with $A^{c/0}$ and $A^*$ where defined: 
that is, $\langle v_0, v_1, u^0_2, w^0_2, v^0_i : i \in I^0 - \{ 0,1,2 \} \rangle$ is a copy of $A^*$; and $\langle v_0, v_1, v^c_i : i \in I^c - \{ 0,1 \} \rangle$ is a copy of $A^{c/0}$. 
For the new edges not involving $u^0_2$ or $w^0_2$, let the colours be the following:
for each $i \in I^0 - \{0,1,2\}$, $j \in I^c - \{0,1\}$ with $i<j$, let $F(v^0_i v^c_j) = \delta^{0c}_{ij} \in C_{ij}$; and 
for each $i \in I^c - \{0,1\}$, $j \in I^0 - \{0,1,2\}$ with $i<j$, let $F(v^c_i v^0_j) = \delta^{c0}_{ij} \in C_{ij}$; note that these colours do not really play a role. 
Finally, for each $j \in I^c - \{0,1,2\}$, let $F(u^0_2 v^c_j) = F(w^0_2 v^c_j) = \gamma^{0c}_{2j}$.

\begin{claim} \label{claimBc*}
For each $c \in \{ 1, \ldots, t-1 \}$, there are colours $\delta^{0c}_{ij}  \in C_{ij}$ for each $i \in I^0 - \{0,1,2\}$, $j \in I^c - \{0,1\}$ with $i<j$; 
$\delta^{c0}_{ij} \in C_{ij}$ for each $i \in I^c - \{0,1\}$, $j \in I^0 - \{0,1,2\}$ with $i<j$;  
and $\gamma^{0c}_{2j} \in C_{2j}$ for each $j \in I^c - \{0,1,2\}$, 
such that $B^{c*}$ is realized in $G$. 
\end{claim}

\begin{proof}
For each $c \in \{ 1, \ldots, t-1 \}$ the idea is to construct the graph $B^{c*}$ by amalgamating $A^{c/0}$ and $A^*$ over their shared 0-coloured $E_{01}$-edge in such a way that for each $j \in I^c - \{0,1,2\}$, we have $F(u^0_2 v^c_j) = F(w^0_2 v^c_j) \in C_{2j}$. 
We always use free amalgamation (that is, other than the vertices of the $E_{01}$-edge, no other vertices are identified in the amalgamation).
We construct $B^{c*}$ in steps by showing that $B^{c*}_{|J|} := B^{c*}|J$ (the restriction of $B^{c*}$ to the parts $\underset{j \in J} \bigcup V_j$) is realized in $G$ for each initial segment $J$ of $I^{c*}$, by induction on the size of $J$. 

Note that while the parts $V_0, V_1, V_2$ are fixed across all of these graphs that we construct (the $B^{c*}$ for each $c \in \{1, \ldots, t-1\}$), for the construction of each one individually, we may change the enumeration of the other parts $V_3, \ldots, V_{m-1}$ so that this step by step process is as `nice' as possible. For instance, take some $c \in \{1, \ldots, t-1\}$, then without loss of generality, but for ease of notation, we may assume that $I^{c*}= \{0,1, \ldots, |I^{c*}|-1 \}$.

The ability to change the enumeration of the parts $V_3, \ldots, V_{m-1}$ for different $c \in \{ 1, \ldots, t-1 \}$ is important in particular when $|I^{c*}|=m$. If possible, we would like to ensure that $m-1 \notin I^c$, and in most cases this can be done immediately. For instance, if $|I^c|<m-1$, then there is certainly at least one part $V_{i_c}$ with $i_c \in \{3, \ldots, m-1\}$ on which $A^c$ is not defined, and so we may assume that $m-1 \notin I^c$. In general, we have two cases to consider. If $A^0$ is a triangle, then since ${\cal A}$ is a star cover set, there is no monic in ${\cal A}$ defined on all of the parts $V_3, \ldots, V_{m-1}$. So for each $c \in \{ 1, \ldots, t-1 \}$, again there is at least one part $V_{i_c}$ with $i_c \in \{3, \ldots, m-1\}$ on which $A^c$ is not defined, and so we may assume that $m-1 \notin I^c$. 

Otherwise, if $A^0$ is defined on at least 4 parts, then we have two further subcases to consider. 
Firstly, if ${\cal A}$ is $(i,j)$-free for some $i,j \in \{2, \ldots, m-1\}$, then recall that we chose part $V_2$ to work over such that $i \ne 2 \ne j$ (and $2 \in I^0$). So ${\cal A}$ is $(i,j)$-free for some $i,j \in \{3, \ldots, m-1\}$, that is, no monic in ${\cal A}$ is defined on both of the parts $V_i$ and $V_j$. Thus again for each $c \in \{ 1, \ldots, t-1 \}$ there is at least one part (other than $V_2$) on which $A^c$ is not defined, and so we may assume that $m-1 \notin I^c$. 

Finally, we are left with the case $(\dag)$ where $A^0$ is defined on at least 4 parts but ${\cal A}$ is not $(i,j)$-free for any choice of $i,j \in \{2, \ldots, m-1\}$. In this case, we cannot yet ensure that there is no monic in ${\cal A}$ defined on all of the parts $V_3, \ldots, V_{m-1}$ (so we may have some $A^c$ with $m-1 \in I^c$); but we shall see that if there is such a monic, this contradicts the minimality of ${\cal A}$.

We now proceed with the induction, which starts with the initial case $|J|=3$, $J = \{ 0,1,2 \}$.
In this case there are no new edges formed in the amalgam $B^{c*}_3$ of $A^*|012$ and $A^c|012$ (the restriction of these graphs to the parts $V_0, V_1, V_2$) over their shared $0$-coloured $E_{01}$-edge.
In particular, if $2 \notin I^c$, then $B^{c*}_3$ is actually just $A^*|012$.

Now for the induction step, suppose we have realized $B^{c*}_n$ in $G$ for some $n$ with $3 \le n < |I^{c*}|$, and we show that we can realize $B^{c*}_{n+1}$ in $G$. 
As described above, without loss of generality, we may assume that $J = \{ 0,1,2,3, \ldots, n \}$.
We amalgamate $A^*|J$ and $A^{c/0}|J$ over $B^{c*}_n$, to obtain $B^{c*}_{n+1}$. 
Note that $B^{c*}_{n+1}$ has at most two more vertices than $B^{c*}_n$: $v^0_{n} \in A^*|J$ (if $n \in I^0$) and $v^c_{n} \in A^{c/0}|J$ (if $n \in I^c$). 
We just need to decide the colours of the new edges incident to these vertices.

Firstly, if $n \in I^0$, then consider the new edges incident to $v^0_n$. These are the edges $v^c_iv^0_n$, for each $i \in I^c$ with $2 \le i < n$. 
To determine the colours of these edges, consider amalgamating $B^{c*}_n$ and $A^*|J$ (which are certainly both realized, by the induction hypothesis, and our earlier construction of $A^*$ which is realized in $G$, respectively) over $A^*|(J-\{n\})$, to obtain some graph $X_{n+1}:= B^{c*}_{n+1} - \{v^c_n\}$, with new edges $v^c_iv^0_n$, for each $i \in I^c$ with $2 \le i < n$.
Then let $\delta^{c0}_{in}:= F_{X_{n+1}}(v^c_iv^0_n) \in C_{in}$ for each $i \in I^c$ with $2 \le i < n$.
Then we shall colour the corresponding edges in $B^{c*}_{n+1}$ as in this amalgam $X_{n+1}$: that is, for each $i \in I^c$ with $2 \le i < n$, let $F_{B^{c*}_{n+1}}(v^c_i v^0_n) = \delta^{c0}_{in} := F_{X_{n+1}}(v^c_i v^0_n)$.
Otherwise, if $n \notin I^0$, then let $X_{n+1}:=B^{c*}_n$.

Next if $n \in I^c$, then first consider the new edges incident to $v^c_n$, but not incident to $u^0_2$ or $w^0_2$. These are the edges $v^0_jv^c_n$, for each $j \in I^0$ with $3 \le j < n$.
To determine the colours of these edges, consider amalgamating $B^{c*}_{n+1} - \{u^0_2, w^0_2, v^c_n\} = X_{n+1} - \{u^0_2, w^0_2\}$ and $A^{c/0}|J$ (which are certainly both realized, by the previous paragraph, and the fact that ${\cal A}$ is good with key colour 0, respectively) over $A^*|(J-\{n\})$, to obtain some graph $Y_{n+1}:= B^{c*}_{n+1} - \{u^0_2, w^0_2\}$, with new edges $v^0_jv^c_n$, for each $j \in I^0$ with $3 \le j < n$.
Then let $\delta^{0c}_{jn}:= F_{Y_{n+1}}(v^0_jv^c_n) \in C_{jn}$ for each $j \in I^0$ with $3 \le j < n$.
Then we shall colour the corresponding edges in $B^{c*}_{n+1}$ as in this amalgam $Y_{n+1}$: that is, for each $j \in I^0$ with $3 \le j < n$, let $F_{B^{c*}_{n+1}}(v^0_jv^c_n) = \delta^{0c}_{jn} := F_{Y_{n+1}}(v^0_jv^c_n)$.

Finally, if $n \in I^c$, then we are just left with determining the colours of the new edges $u^0_2 v^c_n$ and $w^0_2 v^c_n$.

\begin{claim} \label{claimBc*n+1}
For some $\gamma^{0c}_{2n} \in C_{2n}$, the graph $B^{c*}_{n+1}$ with $F(u^0_2 v^c_n) = F(w^0_2 v^c_n) = \gamma^{0c}_{2n}$ (and all other edges coloured as described above) is realized in $G$.
\end{claim}

\begin{proof}[Proof of claim]
Suppose otherwise, and aim for a contradiction.
For each $e \in C_{2n}$, let $Z^e$ denote the graph $B^{c*}_{n+1}$ with $F(u^0_2 v^c_n) = F(w^0_2 v^c_n) = e$ (and all other edges coloured as described above). 
So we assume that for each $e \in C_{2n}$, the graph $Z^e$ is omitted from $G$.
That is, for each $e \in C_{2n}$ some monic subgraph $Y^e$ of $Z^e$ is minimally omitted.

Note that for each $e \in C_{2n}$, all monic subgraphs of $Z^e$ that do not include $v^c_n$ and either $u^0_2$ or $w^0_2$, are certainly realized in $G$ because they must be monic subgraphs of either $Z^e - \{v^c_n\} = X_{n+1}$ or $Z^e - \{u^0_2, w^0_2\} = Y_{n+1}$. So for each $e \in C_{2n}$, the minimally omitted monic $Y^e$ must contain $v^c_n$ and either $u^0_2$ or $w^0_2$. 
That is, for each $e \in C_{2n}$, some monic $Y^e$ with an $e$-coloured $E_{2n}$-edge, is minimally omitted.
Thus ${\cal Y} := \{ Y^e : e \in C_{2n} \}$ is a $C_{2n}$-cover set. 

We show that it is not possible to find such a $C_{2n}$-cover set.
Consider the following graphs: 

$D$ is a $(J- \{2,n\})$-graph with vertices $u_0, w_0 \in V_0$; $u_1, w_1 \in V_1$; 
$v^0_i \in V_i$ for each $i \in I^0$ with $3 \le i \le n$; 
$v^c_i \in V_i$ for each $i \in I^c$ with $3 \le i \le n$; 
such that $D- \{w_0, w_1\}$ and $D- \{u_0, u_1\}$ are both copies of $B^{c*}_n - V_2 = B^{c*}_n - \{u^0_2, w^0_2, v^c_2\}$, and $F(u_0 u_1) = F(w_0 w_1) = 0, ~F(u_0 w_1) = F(w_0 u_1) = d \in C_{01}$ (for some $d \in C_{01}$ to be determined).

$D_2 = D \cup \{v^0_2\}$ (where $v^0_2 \in V_2$) is a $(J- \{n\})$-graph which agrees with $D$ where defined, and such that $D_2 - \{w_0, w_1\}$ is a copy of $B^{c*}_n - \{w^0_2, v^c_2\}$, and $D_2 - \{u_0, u_1\}$ is a copy of $B^{c*}_n - \{u^0_2, v^c_2\}$. 
That is, the new edges are coloured as follows: 
$F(u_0 v^0_2) = \alpha, ~F(w_0 v^0_2) = 0' \in C_{02}$;
$F(u_1 v^0_2) = 0', ~F(w_1 v^0_2) = \beta \in C_{12}$;
$F(v^0_2 v^0_i) = 0' \in C_{2i}$ for each $i \in I^0$ with $3 \le i \le n$;
$F(v^0_2 v^c_i) = \gamma^{0c}_{2i} \in C_{2i}$ for each $i \in I^c$ with $3 \le i \le n$. 

$D_n = D \cup \{v^0_n\}$ (where $v^0_n \in V_n$) is a $(J- \{2\})$-graph which agrees with $D$ where defined, and such that $D_n - \{w_0, w_1\}$ and $D_n - \{w_0, w_1\}$ are both copies of $Y_{n+1} - \{v^c_2, v^0_n\}$. 
That is, the new edges are coloured as follows: 
$F(u_0 v^c_n) = F(w_0 v^c_n) = c' \in C_{0n}$;
$F(u_1 v^c_n) = F(w_1 v^c_n) = c' \in C_{1n}$;
$F(v^0_i v^c_n) = \delta^{0c}_{in} \in C_{in}$ for $i \in I^0$ with $3 \le i \le n$; 
$F(v^c_i v^c_n) = c' \in C_{in}$ for $i \in I^c$ with $3 \le i \le n$. 

\begin{claim} \label{claimD}
For some $\zeta \in C_{01}$, the graphs $D, D_2, D_n$ with $F(u_0 w_1) = F(w_0 u_1) = \zeta$ (and all other edges coloured as described above) are all realized in $G$.
\end{claim}

\begin{proof}[Proof of claim]
Recall that by Corollary~\ref{monic}, a graph is realized in $G$ if and only if all its monic subgraphs are realized in $G$. So we aim to show that for some $\zeta \in C_{01}$ all monic subgraphs of $D, D_2, D_n$ are realized in $G$.

First observe that all monic subgraphs which do not contain both $u_0, w_1$, or both $w_0, u_1$, are certainly realized, since such monics are also (copies of) subgraphs of $B^{c*}_n - \{w^0_2, v^c_2\}$, $B^{c*}_n - \{u^0_2, v^c_2\}$, or  $Y_{n+1} - \{v^c_2, v^0_n\}$, and we know that $B^{c*}_n, Y_{n+1}$ are certainly both realized (by the induction hypothesis, and the construction preceding Claim~\ref{claimBc*n+1}).

So if $D, D_2, D_n$ are not all realized for any $d \in C_{01}$, then we can find a $C_{01}$-cover set ${\cal D} = \{D^d: d \in C_{01} \}$ such that each $D^d$ is a monic subgraph of $D_2$ or $D_n$ containing either $u_0, w_1$ or $w_0, u_1$ (so $D^d$ has a $d$-coloured $E_{01}$-edge).
However we shall show that no such $C_{01}$-cover set can exist. 

Suppose that there is a $C_{01}$-omission set ${\cal D}^*$ based on ${\cal D}$. Observe that it must have a code $(0,1,2,l; \alpha, c_{0l}, \beta, c_{1l})$ for some $l$ with $3 \le l \le n$, otherwise ${\cal D}^*$ is also a $C_{01}$-omission set based on ${\cal A}$ (contradicting the fact that ${\cal A}$ was chosen such that no such omission set exists). 
But note that $T^{0 \alpha \beta}$ (the $012$-triangle with $0$-coloured $E_{01}$-edge, $\alpha$-coloured $E_{02}$-edge, and $\beta$-coloured $E_{12}$-edge) is realized in $G$; 
for each $i \in I^0$ with $3 \le i \le n$ the $01i$-triangle $T^0_i$ with $0$-coloured $E_{01}$-edge, $0'$-coloured $E_{0i}$-edge, and $0'$-coloured $E_{1i}$-edge is a proper subgraph of minimally omitted $A^0$ and so is realized in $G$; 
and for each $i \in I^c$ with $3 \le i \le n$ the $01i$-triangle $T^c_i$ with $0$-coloured $E_{01}$-edge, $c'$-coloured $E_{0i}$-edge, and $c'$-coloured $E_{1i}$-edge is a subgraph of $A^{c/0}$ which is realized in $G$ since ${\cal A}$ is a good cover set. So no such omission set ${\cal D}^*$ exists.

Thus ${\cal D}$ is a $C_{01}$-cover set with no omission set based on it, but we shall see that this contradicts the choice of ${\cal A}$ as a minimal such $C_{01}$-cover set.
Observe that ${\cal D}$ is a $C_{01}$-cover set in $\Lambda$ defined on $n+1$ parts, containing no monics defined on both $V_2$ and $V_n$, that is, ${\cal D}$ is $(2,n)$-free. 
Note that if $n+1 < m$, this clearly contradicts the choice of $G$ and $m$ (recall that $m$ was taken to be the minimum size possible such that there is a homogeneous $m$-generic graph $G$ such that $O(G)$ contains a $C_{01}$-cover set, but no omission set based on it; and then certainly such a cover set is defined on all $m$ parts). 

Otherwise, we are in the case that $n+1 = m$ and $n = m-1 \in I^c$. Now we must be in the case $(\dag)$, that is, where ${\cal A}$ is not $(i,j)$-free for any choice of $i,j \in \{2, \ldots, m-1\}$ (and $A^0$ is defined on at least 4 parts)---recall the discussion at the beginning of the proof of Claim~\ref{claimBc*}. But now ${\cal D}$ is a $C_{01}$-cover set in $\Lambda$ with no omission set based on it, which is defined on $m$ parts and is $(2,m-1)$-free. This contradicts the initial choice of ${\cal A}$ as minimal, so in fact this case cannot arise.

Hence there is no such $C_{01}$-cover set ${\cal D}$ in $O(G)$, and so there is some $\zeta \in C_{01}$ such that all monic subgraphs of $D, D_2, D_n$ with $F(u_0 w_1) = F(w_0 u_1) = \zeta$ are realized in $G$. Thus by Corollary~\ref{monic}, $D, D_2, D_n$ with $F(u_0 w_1) = F(w_0 u_1) = \zeta$ are realized in $G$. So Claim~\ref{claimD} is proved.
\end{proof}

Now amalgamate $D_2$ and $D_n$ over $D$ to obtain $D^c$. The only new edge is $v^0_2 v^c_n$, and suppose that $F_{D^c}(v^0_2 v^c_n) = \varepsilon \in C_{2n}$. 

Recall that $Z^\varepsilon$ is the graph $B^{c*}_{n+1}$ with $F(u^0_2 v^c_n) = F(w^0_2 v^c_n) = \varepsilon$. Observe that $D^c - \{ w_0, w_1\}$ is isomorphic to $Z^\varepsilon - \{w^0_2, v^c_2, v^0_n \}$, and $D^c - \{ u_0, u_1\}$ is isomorphic to $Z^\varepsilon - \{u^0_2, v^c_2, v^0_n \}$.
But now note that $Y^\varepsilon$ is a monic subgraph (containing $v^c_n$ and either $u^0_2$ or $w^0_2$) of one of these subgraphs of $Z^\varepsilon$.
Then $Y^\varepsilon$ is realized in $D^c$ (which we have realized in $G$), which contradicts the assumption that we could find the $C_{2n}$-cover set ${\cal Y}$.

Thus there is no such $C_{2n}$-cover set, and so we can find $\gamma^{0c}_{2n} \in C_{2n}$ such that $B^{c*}_{n+1}$ with $F(u^0_2 v^c_n) = F(w^0_2 v^c_n) = \gamma^{0c}_{2n}$ (and all other edges coloured as previously described) is realized in $G$ (for instance let $\gamma^{0c}_{2n}:= \varepsilon$), as required. So Claim~\ref{claimBc*n+1} is proved.
\end{proof}

Thus by induction we construct the graph $B^{c*}$ which is realized in $G$, which is an amalgam of $A^{c/0}$ and $A^*$ over their shared 0-coloured $E_{01}$-edge such that for each $j \in I^c - \{0,1,2\}$, we have $F(u^0_2 v^c_j) = F(w^0_2 v^c_j) \in C_{2j}$. So Claim~\ref{claimBc*} is proved.
\end{proof}

We now form an $m$-partite graph which is realized in $G$ by amalgamating all of the graphs $B^{c*}$, for each $c \in \{1, \ldots, t-1\}$, over $A^*$. Note that at this stage we need to fix the enumeration of all of the parts, because we consider the set of all of the $B^{c*}$. We may perform the amalgamations one by one: first amalgamate $B^{1*}$ and $B^{2*}$, then amalgamate the result with $B^{3*}$, and so on up to $B^{t-1*}$. Each time we may use free amalgamation (that is, we do not identify vertices other than those in $A^*$) and this is straightforward. We call the resulting $m$-partite graph $B^*$.

Now consider $H_0 = B^* - \{v_1, u^0_2 \}$ (an $I_0$-graph with $I_0 = \{0,2,3,\dots,m-1\}$) and $H_1 = B^* - \{v_0, w^0_2 \}$ (an $I_1$-graph with $I_1 = \{1,2,3,\dots,m-1\}$). 
Observe that $H_0 - \{v_0\}$ and $H_1 - \{v_1\}$ (both $I_{2}$-graphs with $I_2 = \{2,3,\dots,m-1\}$) are isomorphic: consider the map which sends $u^0_2$ to $w^0_2$ and fixes all other vertices; this is an isomorphism since for any $x \in B^*|I_3$ with $I_3 = \{3,4,\ldots, m-1\}$, by our construction $F_{B^*}(u^0_2 x) = F_{B^*}(w^0_2 x)$.

Finally amalgamate $H_0$ and $H_1$ over this common substructure to form $H$ which is realized in $G$. In this amalgamation $u^0_2 \in H_1$ and 
$w^0_2 \in H_0$ are identified, and we relabel this single vertex $v^0_2$. The only new edge of $H$ is $v_0 v_1 \in E_{01}$, say 
$F_H(v_0 v_1) = \gamma \in C_{01}$. Let $\delta \in M$ be such that $A^\gamma \le_{01} A^\delta$. But now consider the monic subgraph 
$\langle v_0, v_1, v^\delta_i :i \in I^\gamma - \{0,1\} \rangle$ of $H$. This is a realization of $A^\gamma$ in $G$---which contradicts the assumption 
that $A^c$ was minimally omitted for each $c \in C_{01}$.

This final contradiction means that our initial assumption that there was no $C_{01}$-omission set based on the minimal star $C_{01}$-cover set 
${\cal A}$ must be wrong, and hence we have finished. 
\end{proof}

Thus we have proved the non-complication theorem for $m$-generic graphs. It remains to fully define and verify the GCSSA, and show that it can be used to find star cover sets as claimed (Lemma~\ref{starcoverset}).

\subsection{The good cover set search algorithm} \label{sectionGCSSA}

Recall the outline of the algorithm as described before Lemma~\ref{goodcoverset}, which we now expand. The input for the algorithm is any cover set 
${\cal B}_0 \in \Lambda$. At each step $i$, a particular maximal member $B^{\gamma_i}$ of the cover set ${\cal B}_i \in \Lambda$ is considered (that is, a maximal member of the quasi-order $({\cal B}_i, \le_{01})$). If $B^{\gamma_i}$ (which is called the \emph{test monic} at this step) is a key monic, then ${\cal B}_i$ is good, and the algorithm stops. Otherwise another minimally omitted monic $D^{\gamma_i} \in O(G)$ is found such 
that $D^{\gamma_i} \subseteq B^{\delta_i / \gamma_i}$ for some suitable $\delta_i \in C_{01}$, $\delta_i \ne \gamma_i$, with 
$B^{\delta_i} \in {\cal B}_i$, and $B^{\gamma_i}$ is replaced by $D^{\gamma_i}$ in the next step to obtain ${\cal B}_{i+1} \in \Lambda$ (note that 
${\cal B}_{i+1} \preceq {\cal B}_i$). 
By the careful choice of the test monic $B^{\gamma_i}$ at each step, if the algorithm does not find a good cover set, then we shall show that it still stops after a finite number of 
steps, and when it does, the cover set ${\cal D}$ at that step has one fewer maximal element (up to off-$E_{01}$-isomorphism) than ${\cal B}_0$.

In the algorithm, the most important characteristic of the cover sets ${\cal B}_i$ is what they look like as quasi-orders. To see the action of the algorithm on an initial cover set ${\cal B}_0$, we picture the sequence of quasi-orders $({\cal B}_0, \le_{01}), ({\cal B}_1, \le_{01}), \ldots$ that are produced at subsequent steps. At each step, there are exactly $|C_{01}|$ elements in the quasi-order (for each $c \in C_{01}$, exactly one element has a $c$-coloured $E_{01}$-edge). From step $i$ to step $i+1$, exactly one element `moves' in the quasi-order, as $B^{\gamma_i}$ is replaced by $D^{\gamma_i}$ (note that these monics have the same colour $E_{01}$-edge). That is, if the maximal element $B^{\gamma_i}$ is not a key monic, then it is replaced by another minimally omitted monic $D^{\gamma_i}$ which is an off-$E_{01}$-subgraph of some other maximal element $B^{\delta_i}$ of the cover set.

To ensure that the algorithm works as asserted, it is important that at each step $i$ the test monic $B^{\gamma_i}$ is carefully chosen. Thus the test monics are always taken from the same `part' of the quasi-order, which is determined by the choice of the first test monic $B^{\gamma_0}$. Specifically, the test monics are the members of the cover set which are off-$E_{01}$-subgraphs of maximal $B^{\gamma_0}$ but not off-$E_{01}$-subgraphs of any other maximal element of the cover set (we call this `part' the \emph{test set}). 
We make replacements using the property that at that step the cover set is not good with the test monic as a key monic. By Remark~\ref{key}, at each step the test monic must be a maximal element of the cover set at that step, and the minimally omitted monic which we replace it by will be an off-$E_{01}$-subgraph of some other maximal element of the cover set at that step. 
By working through the members of the test set in a prescribed order (namely where possible choosing $B^{\gamma_{i+1}} \le_{01} B^{\gamma_i}$), we shall see that the same monic is never tested twice, so since there are only finitely many monics in $O(G)$, the algorithm must stop. In particular, if a good cover set is not found, then in a finite number of steps all members of the test set will have been moved to a different part of the quasi-order, and so the algorithm stops with final cover set ${\cal D} \in \Lambda$ which indeed has one fewer maximal element (up to off-$E_{01}$-isomorphism) than the initial cover set ${\cal B}_0$.

With this intuition in mind, let us now formally describe the algorithm. 

\begin{alg} 
The good cover set search algorithm (GCSSA).

\noindent \textbf{Initial step $0$:}
Let ${\cal B}_0$ be a cover set in $\Lambda$. 
Pick a colour $\gamma_0 \in C_{01}$ such that $B^{\gamma_0}_0$ is a maximal element of ${\cal B}_0$. 
Let $\Delta_0 := \{ c \in C_{01} : \textrm{ for all } d \in C_{01}, \textrm{ if } B^c_0 \le_{01} B^d_0, \textrm{ then } B^d_0 \le_{01} B^{\gamma_0}_0 \}$. 

If ${\cal B}_0$ is good with $B^{\gamma_0}_0$ as a key monic, let ${\cal B} := {\cal B}_0$ and stop. 

Otherwise, $\gamma_0$ is not a key colour for the cover set ${\cal B}_0$, so there is some colour $\delta_0 \in C_{01}$ with $B^{\delta_0}_0 \not\le_{01} B^{\gamma_0}_0$ such that $B^{\delta_0/\gamma_0}_0$ is not realized in $G$ (note in particular that $\delta_0 \in C_{01} - \Delta_0$). 
Let $D^{\gamma_0}_0 \in O(G)$ be a minimally omitted subgraph of $B^{\delta_0/\gamma_0}_0$. 
Note that $D^{\gamma_0}_0$ is certainly defined on $V_0$ and $V_1$, since $D^{\gamma_0}_0 \subseteq B^{\delta_0/\gamma_0}_0$ and $B^{\delta_0}_0$ is minimally omitted. Go to step $1$.

\sk
\noindent \textbf{Iterative step:}
We now describe the process for going from step $i$ to step $i+1$ in the algorithm. So first of all, we state what we have at the end of step $i$. 

After step $i$ we have the following: a cover set ${\cal B}_i = \{ B^c_i : c \in C_{01} \} \in \Lambda$; a subset of the colours $\Delta_i \subset C_{01}$;
a distinguished colour $\gamma_i \in \Delta_i$ such that $B^{\gamma_i}_i$ is a maximal element of ${\cal B}_i$; 
another distinguished colour $\delta_i \in C_{01}$ (such that $B^{\delta_i}_i \not\le_{01} B^{\gamma_i}_i$), and a monic $D^{\gamma_i}_i \in O(G)$ (with $\gamma_i$-coloured $E_{01}$-edge) which is a minimally omitted subgraph of $B^{\delta_i/\gamma_i}_i$.

\sk
\noindent \textbf{Step $i+1$:} 
Construct a new cover set ${\cal B}_{i+1} := \{ B^c_{i+1} : c \in C_{01} \}$, where $B^{\gamma_i}_{i+1} := D^{\gamma_i}_i$, and $B^c_{i+1} := B^c_i$ for each $c \in C_{01} - \{\gamma_i\}$. 
Note that ${\cal B}_{i+1} \in \Lambda$, $B^{\gamma_i}_{i+1} \le_{01} B^{\delta_i}_{i+1}$, and ${\cal B}_{i+1} \preceq {\cal B}_i$.

Now let 
\[ \Delta_{i+1} := \left\{ \begin{array}{ll}
\Delta_i & \textrm{if $\delta_i \in \Delta_i$}\\
\Delta_i - \{ \gamma_i \} & \textrm{if $\delta_i \not\in \Delta_i.$}
\end{array} \right.
\] 

If $\Delta_{i+1} = \emptyset$, then let ${\cal D} := {\cal B}_{i+1}$ and stop.
Otherwise, choose $\gamma_{i+1} \in \Delta_{i+1}$ such that $B^{\gamma_{i+1}}_{i+1}$ is a maximal element of $({\cal B}_{i+1}, \le_{01})$, and if possible $B^{\gamma_{i+1}}_{i+1} = B^{\gamma_{i+1}}_i \le_{01} B^{\gamma_i}_i$.

If ${\cal B}_{i+1}$ is good with $B^{\gamma_{i+1}}_{i+1}$ as a key monic, then let ${\cal B} := {\cal B}_{i+1}$ and stop. 

Otherwise, $\gamma_{i+1}$ is not a key colour for the cover set ${\cal B}_{i+1}$, so there is some colour $\delta_{i+1} \in C_{01}$ with $B^{\delta_{i+1}}_{i+1} \not\le_{01} B^{\gamma_{i+1}}_{i+1}$ such that $B^{\delta_{i+1}/\gamma_{i+1}}_{i+1}$ is not realized in $G$. 
If possible, find such $\delta_{i+1} \in C_{01} - \Delta_{i+1}$ (this will be a `good replacement'). 
Let $D^{\gamma_{i+1}}_{i+1} \in O(G)$ be a minimally omitted subgraph of $B^{\delta_{i+1}/\gamma_{i+1}}_{i+1}$. 
(Note that as in step $0$, the monic $D^{\gamma_{i+1}}_{i+1}$ is certainly defined on $V_0$ and $V_1$, since $D^{\gamma_{i+1}}_{i+1} \subseteq B^{\delta_{i+1}/\gamma_{i+1}}_{i+1}$ and $B^{\delta_{i+1}}_{i+1}$ is minimally omitted.)
Go to step $i+2$.
\end{alg}

\sk
We verify that the good cover set search algorithm works as stated. 

\begin{lem}
Given a homogeneous $m$-generic graph $G$ and a $C_{01}$-cover set ${\cal B}_0 \in \Lambda$ in $O(G)$, the good cover set search algorithm either stops when it finds a good $C_{01}$-cover set ${\cal B} \in \Lambda$ with ${\cal B} \preceq {\cal B}_0$, or when it finds a $C_{01}$-cover set ${\cal D} \in \Lambda$ with ${\cal D} \prec {\cal B}_0$ and with one fewer maximal element (up to off-$E_{01}$-isomorphism) than ${\cal B}_0$ (that is, $t_{\cal D} = t_{{\cal B}_0} - 1$). 
\end{lem}

\begin{proof}
We consider applying the GCSSA to the $C_{01}$-cover set ${\cal B}_0 = \{ B^c_0 : c \in C_{01} \} \in \Lambda$. 
We aim to show that the formal algorithm works as asserted, by explaining the different parts of the steps, and verifying the action of these, as 
necessary.

In the initial step, an arbitrary maximal element $B^{\gamma_0}_0$ of ${\cal B}_0$ is chosen, and then the set $\Delta_0 \subset C_{01}$ is defined 
using $\gamma_0$. The set $\Delta_0$ is used to define the `part' of the quasi-order $({\cal B}_0, \le_{01})$ from which all test monics (for 
subsequent steps) will be taken (we shall see that the test monic at each step has a $c$-coloured $E_{01}$-edge for some $c \in \Delta_0$). This `part' is the subset ${\cal B}^{\Delta}_0 := \{B^c \in {\cal B}_0 : c \in \Delta_0 \}$ of ${\cal B}_0$, which (by the definition of $\Delta_0$) consists of all members of ${\cal B}_0$ which are off-$E_{01}$-subgraphs of $B^{\gamma_0}_0$ but not off-$E_{01}$-subgraphs of any other maximal element of ${\cal B}_0$ (we note that ${\cal B}^{\Delta}_0$ includes $B^{\gamma_0}_0$ itself, the first test monic).
(We may note that by Lemma~\ref{mincoverset}, ${\cal B}_0$ has at least two non-off-$E_{01}$-isomorphic maximal elements. So $\Delta_0$ is a proper non-empty subset of $C_{01}$.)

The rest of the inital step, which consists of testing whether $B^{\gamma_0}_0$ is a key monic, and if it is not, the definition of the colour $\delta_0$ and the monic $D^{\gamma_0}_0 \in O(G)$, should be straightforward, by Remark~\ref{key}.

We now consider the iterative step. 
Firstly a new cover set ${\cal B}_{i+1}$ is constructed from the previous one by replacing the monic $B^{\gamma_i}_i$ (which was not a key monic of ${\cal B}_i$) by the monic $D^{\gamma_i}_i \in O(G)$ (which is renamed $B^{\gamma_i}_{i+1}$). All other members of the cover set remain the same, except they are relabelled by the new subscript. It is clear that ${\cal B}_{i+1} \in \Lambda$, and ${\cal B}_{i+1} \preceq {\cal B}_i$ since $B^{\gamma_i}_{i+1} \le_{01} B^{\delta_i}_{i+1}$ (that is, $D^{\gamma_i}_i \le_{01} B^{\delta_i}_i$, which is obvious since $D^{\gamma_i}$ was chosen to be a (minimally omitted) subgraph of $B^{\delta_i/\gamma_i}_i$) and $B^c_{i+1} = B^c_i$ for all $c \in C_{01} - \{\gamma_i\}$.

The purpose of the rest of the iterative step is to first choose the next test monic $B^{\gamma_{i+1}}_{i+1}$ (which must be done carefully to ensure that we never test the same monic twice), and second, to test whether this is a key monic of the cover set ${\cal B}_{i+1}$, and if not to choose the colour $\delta_{i+1}$ and define $D^{\gamma_{i+1}}_{i+1}$. The second part (testing $B^{\gamma_{i+1}}_{i+1}$, defining $\delta_{i+1}$ and $D^{\gamma_{i+1}}_{i+1}$) is straightforward by Remark~\ref{key}, and works exactly as in the initial step. It remains to explain and verify the first part (choosing the next test monic).

Ideally, we would simply work through the subset ${\cal B}^{\Delta}_0$ (which we call the \emph{test set}) considering members one by one, testing whether any is a key monic, stopping if it is, and if it is not simply removing it from the set and moving on to the next test monic (choosing one which is maximal in the remaining subset). However, as we work though the algorithm, monics may be added to this test set as well as removed. This occurs precisely when it is not possible to make a `good replacement', and the test monic $B^{\gamma_i}_i$ must be replaced by $D^{\gamma_i}_i$ in the test set. Thus the test set at step $i$ for $i > 0$ is not necessarily simply a subset of the initial test set ${\cal B}^{\Delta}_0$. However, we do get a non-increasing sequence (which is preferable) when we only consider the sets of $E_{01}$-edge colours of the members of the test sets (rather than the test sets themselves). Thus we define the set $\Delta_{i+1}$ from $\Delta_i$, by either leaving it unchanged, or removing exactly one colour; and these colour sets determine the test sets.

In more detail, if $B^{\gamma_i}_i$ is not a key monic at step $i$,  then we make a replacement, aiming (if possible) to remove the colour $\gamma_i$ from the set $\Delta_i$, so that as often as possible this set decreases. At each step $i+1$ we either remove $\gamma_i$ from $\Delta_i$ to obtain $\Delta_{i+1}$, or leave the set unchanged (so $\Delta_{i+1} = \Delta_i$). This depends on whether or not we could make a `good replacement'---that is, whether we could find $\delta_i \in C_{01} - \Delta_i$ such that $B^{\delta_i/\gamma_i}_i$ is not realized in $G$, or if there was only such a $\delta_i \in \Delta_i$. It is easy to see that $\Delta_{i+1} \subseteq \Delta_i$ for all $i \ge 0$, and this is a proper subset precisely when we could make a good replacement. Then $\Delta_i \subseteq \Delta_0$ for all $i \ge 0$, and so the test monic at each step indeed has a $c$-coloured $E_{01}$-edge for some $c \in \Delta_0$.

We keep track of the part of the quasi-order from which the test monics are always chosen by using the $\Delta_i$ sets. At each step we must certainly choose a test monic from the test set which is maximal in the cover set, and we now explain how this is done in a specified way to prevent the algorithm from testing the same monic twice. 
At each step $i$, the test monic $B^{\gamma_i}_i$ is either a member of ${\cal B}^\Delta_0$ from the original cover set, or it is a replacement monic that was added at some previous step when no good replacement was possible.
Whenever possible we choose $\gamma_{i+1} \in \Delta_{i+1}$ such that $B^{\gamma_{i+1}}_i \le_{01} B^{\gamma_i}_i$, and $B^{\gamma_{i+1}}_{i+1} = B^{\gamma_{i+1}}_i$ is maximal in $({\cal B}_{i+1}, \le_{01})$. In the next paragraph we explain how by doing this, we remove the possibility of ever using a test monic from a previous step as a replacement monic. In particular, this means that a monic is never tested twice. When there is no such member of the test set below $B^{\gamma_i}_i$ in the quasi-order $({\cal B}_i, \le_{01})$, we simply choose any maximal element of $({\cal B}_{i+1}, \le_{01})$ as the next test monic.

Suppose that the test monic at step $k$ was $B^{\alpha}_k$, and we are now at step $l>k$ with test monic $B^{\beta}_l$ which is not a key monic. Suppose that $B^{\alpha/\beta}_k \in O(G)$ so if there exists $B^c_l$ with $B^{\alpha}_k \le_{01} B^c_l$ then we may consider $B^{\alpha/\beta}_k$ as a replacement monic at step $l$. But since $B^{\alpha}_k$ was the test monic at step $k$ it must have been maximal, so there are certainly no such $B^c_k$ with $B^{\alpha}_k <_{01} B^c_k$, and no such monics are introduced at later steps (a replacement monic is never strictly above any existing maximal element in the quasi-order, because it must be an off-$E_{01}$-subgraph of some element; so the quasi-order cannot `grow upwards'). Also, there are no such $B^c_l$ with $B^{\alpha}_k =_{01} B^c_l$, because such a monic must have been present in the cover set ${\cal B}_k$ (since the quasi-order cannot `grow upwards'), that is there existed $B^c_k = B^c_l =_{01} B^{\alpha}_k$; but then since we choose test monics which are maximal such that $B^{\gamma_{i+1}}_i \le_{01} B^{\gamma_i}_i$ if possible, $B^c_k$ would have been chosen as the test monic at a step after $k$, before step $l$. Having reached step $l$, we did not stop at such a step, so at that step the monic $B^c_k$ was removed and replaced; so it can't still be in the cover set at step $l$.

Since there are only finitely many monics in $O(G)$ and the algorithm never tests a monic twice, it must stop after a finite number of steps.
So either the algorithm stops by finding a good cover set ${\cal B}$, or after a finite number of steps we reach $\Delta_{i+1} = \emptyset$ and the algorithm stops with cover set $\cal D$. In the latter case, we have run out of test monics without finding a key monic for a good cover set, and the members of the initial test set ${\cal B}^{\Delta}_0$ have all been replaced by good replacements. That is, the maximal element $B^{\gamma_0}_0$ and all elements below  (or off-$E_{01}$-isomorphic to) $B^{\gamma_0}_0$ in the quasi-order $({\cal B}_0, \le_{01})$ but not below any other maximal element, have been removed and replaced by monics below (or off-$E_{01}$-isomorphic to) some other maximal element. Thus the final cover set $\cal D$ has one fewer maximal element (up to off-$E_{01}$-isomorphism) than the initial cover set (that is, $t_{\cal D} = t_{{\cal B}_0} - 1$), as asserted. (We may note that if the algorithm does not find a good cover set, then at each step we could find a good replacement for the test monic, and so the algorithm stops at step $|\Delta_0|$.) 
\end{proof}

Finally we can finish things off by showing how to use the good cover set search algorithm to find a star cover set based on any cover set.

\begin{proof}[Proof of Lemma~\ref{starcoverset}]
Let $G$ be a homogeneous $m$-generic graph which has a $C_{01}$-cover set ${\cal A}$ in $O(G)$. We aim to show that there is a star $C_{01}$-cover set in $\Lambda$ based on ${\cal A}$. 

By Lemma~\ref{goodcoverset}, we may certainly find a good cover set ${\cal B}$ with ${\cal B} \preceq {\cal A}$ using the GCSSA. If ${\cal B}$ is not a star cover set, then it must have a key monic $B^\alpha$ defined on exactly 3 parts with at least one other monic defined on all of the remaining parts, let $B^\beta \in {\cal B}$ be such a monic. When we have good cover sets which are not star cover sets, the problem is these `large' monics. However we may use the GCSSA to get rid of them. 

By applying the GCSSA to ${\cal B}$ starting with $\gamma_0 := \beta$, we find a new cover set ${\cal C}$ based on ${\cal A}$. Either ${\cal C}$ is a good cover set with key monic off-$E_{01}$-isomorphic to $B^\beta$ (possibly $B^\beta$ itself), or ${\cal C}$ has no members off-$E_{01}$-isomorphic to $B^\beta$ (observe that the algorithm works by testing and replacing any such monics first). In the former case, since $B^\beta$ has size at least 4 (since $m \ge 5$), then ${\cal C}$ is in fact a star cover set. Otherwise in the latter case, note that ${\cal C}$ is not necessarily good, but what is important is that ${\cal C}$ is based on ${\cal A}$ and in a suitable sense has fewer large monics. If ${\cal C}$ is a good but not a star cover set, then we may repeat the process. Otherwise, if ${\cal C}$ is a star cover set then we can stop; while if ${\cal C}$ is not good, then by Lemma~\ref{goodcoverset}, we may find a good cover set based on ${\cal C}$ using the GCSSA, and proceed using this cover set. If necessary, this process of using the GCSSA to get rid of large monics can be repeated as many times as is needed to obtain a cover set ${\cal E} \in \Lambda$ based on $\cal A$ such that each member is defined on at most $m-2$ parts. Now any good cover set based on $\cal E$ (which may be found using the GCSSA) must be a star cover set.
\end{proof}

\section{Characterization of $m$-generic graphs}

The conditions from the previous section are enough to characterize all the homogeneous coloured $m$-generic graphs. 

\begin{thm} \label{mgeneric}
Let $L$ be a language for coloured $m$-partite graphs, and let ${\cal F}$ be a family of monic $L$-graphs. Then there is a (unique) countable homogeneous $m$-generic $L$-graph $G$ with $Age(G) = Forb({\cal F})$ if and only if ${\cal F}$ satisfies the following: 
\begin{enumerate}
\item[(i)] if ${\cal A} \subseteq {\cal F}$ is a $C_{ij}^{kl}$-omission set, then there is a corresponding $C_{kl}^{ij}$-omission set  ${\cal B}$ in ${\cal F}$;
\item[(ii)] if ${\cal A} \subseteq {\cal F}$ is a $C_{ij}$-cover set, then for some distinct $k,l$ there is a $C_{ij}^{kl}$-omission set in ${\cal F}$ based on triangles from ${\cal A}$.
\end{enumerate}
\end{thm}

\begin{proof}
The forward direction is given by Lemma~\ref{corros}, Theorem~\ref{mgen} and Lemma~\ref{triangleomset}.

For the converse, we must show that if ${\cal F}$ satisfies conditions (i) and (ii), then $Forb({\cal F})$ is an amalgamation class. 

So take $A, B_1, B_2 \in Forb({\cal F})$ such that $A$ embeds in $B_1$ and $B_2$. We may assume that $B_1 = A \cup \{x\}, ~B_2 = A \cup \{y\}$. Now we just need to decide the colour of the edge $xy$ to form the amalgam $C = A \cup \{x,y\}$ so that no member of ${\cal F}$ is realized. As in the proof of Theorem~\ref{4generic}, the only difficult case is when $x \in V_i, y \in V_j$ and ${\cal F}$ contains a $C_{ij}$-cover set, so consider this case. 

If there is some colour $c \in C_{ij}$ such that $C$ with $F(x,y)=c$ omits all members of ${\cal F}$, then we have finished. 
Otherwise, for each $c \in C_{ij}$, some $A^c \in {\cal F}$ with $c$-coloured $E_{ij}$-edge would be realized in $C$. We aim to show that this will never happen.

Note that the set ${\cal A} := \{ A^c : c \in C_{ij} \}$ is a $C_{ij}$-cover set in ${\cal F}$. Then by (ii), for some distinct $k,l$ there is a 
$C_{ij}^{kl}$-omission set based on triangles from ${\cal A}$, so we can in fact find such $A^c$ which make up a $C_{ij}^{kl}$-omission set. 
That is, there are $v_k, v_l \in A$ such that there is a $C_{ij}^{kl}$-omission set $S_{ij} \subseteq {\cal F}$ with its 
$E_{ik}, E_{il}, E_{jk}, E_{jl}$-edges coloured as in $\langle x, y, v_k, v_l \rangle$. But then by (i), ${\cal F}$ contains a corresponding 
$C_{kl}^{ij}$-omission set $S_{kl}$. Consider the edge $v_kv_l \in A$, say $F(v_k, v_l) = d \in C_{kl}$. At least one of $\langle x, v_k, v_l \rangle$ 
or $\langle y, v_k, v_l \rangle$ with $F(v_k, v_l) = d$ is contained in $S_{kl}$. But $x, v_k, v_l \in B_1$ and $y, v_k, v_l \in B_2$, so then one of 
$B_1, B_2$ realizes a forbidden monic, which contradicts our initial assumptions. 
\end{proof}

The explicit description of the possible values of $O(G)$ given by Theorem~\ref{mgeneric} seems preferable to the general effectiveness argument 
following Theorem~\ref{nonmonic}, though in more complicated classifications, that may be all one can hope for. The classification of all the (not 
necessarily $m$-generic) homogeneous $m$-partite graphs is read off using Lemma~\ref{perfectm} as in section 2.2. It may also be possible to 
address issues of complexity; the complexity of the method presented here is rather crude (just list all possibilities, and check which of them are 
compatible with the non-complication theorem).


\begin{thebibliography}{99}
\bibitem{cherlin} 
G. Cherlin, 
The classification of countable homogeneous directed graphs and countable
$n$-tournaments, 
Mem. Amer. Math. Soc. 131 (1998), 621.

\bibitem{jenkinson1}
T.~Jenkinson, 
The construction and classification of homogeneous structures in model theory, 
PhD thesis, University of Leeds, 2006.

\bibitem{jenkinson2}
T.~Jenkinson, D.~Seidel and  J.~K.~Truss, 
Countable homogeneous multipartite graphs, 
European J.\ Combin.\ 33 (2012), 82--109. 

\bibitem{lachlan1} 
A. H. Lachlan and R. Woodrow, 
Countable ultrahomogeneous undirected graphs, Trans. Amer. Math.
Soc., 262(1): 51-94, 1980.

\bibitem{lachlan2} 
A. H. Lachlan, 
Countable homogeneous tournaments, 
Trans. Amer. Math. Soc. 284 (1984), 431-461.

\bibitem{rose} 
Simon Rose, 
Classification of countable homogeneous 2-graphs, 
PhD thesis, University of Leeds, 2011.

\bibitem{schmerl} J. H. Schmerl, 
Countable homogeneous partially ordered sets, 
Algebra Universalis, 9 (1979) 317-321.

\bibitem{seidel}
D.~Seidel, 
Classification of the countable homogeneous multipartite graphs, 
Diplomarbeit, Technische Universit\" at Freiburg, 2008.

\bibitem{torrezao} 
S. Torrez\~ao de Sousa and J. K. Truss, 
Countable homogeneous coloured partial orders,  
Dissertationes Mathematicae 455 (2008).

\end{thebibliography}
\end{document}